\documentclass{amsart}
\usepackage{tikz}
\usepackage[mathscr]{eucal}
\usepackage{amssymb, mathabx,stmaryrd}
\usepackage{algorithm,algpseudocode}
\usepackage{enumitem}

\usepackage{xcolor}

\newtheorem{theorem}{Theorem}[section]
\newtheorem{lemma}[theorem]{Lemma}
\newtheorem{cor}[theorem]{Corollary}

\newtheorem{defn}[theorem]{Definition}

\newtheorem{example}[theorem]{Example}

\newcommand{\deq}{\mathrel{\mathop:}=}
\newcommand{\alg}[1]{\mathbf{#1}}
\newcommand{\var}[1]{\mathsf{#1}}
\newcommand{\algov}[1]{\overline{\alg{#1}}}
\newcommand{\ov}[1]{\overline{#1}}
\newcommand{\FV}{\alg{F}_\mathcal{V}}
\newcommand{\FW}{\alg{F}_\mathcal{W}}
\newcommand{\FVhat}{\alg{F}_{\widehat{\mathcal{V}}}}
\newcommand{\FVo}{\alg{F}_{\mathcal{V}_0}}
\newcommand{\FVohat}{\alg{F}_{\widehat{\mathcal{V}_0}}}

\newcommand{\Con}[1]{\mathbf{Con}\, #1}
\newcommand{\con}[1]{\mathrm{Con}\, #1}
\newcommand{\Cm}[1]{\mathbf{Cm}\, #1}
\newcommand{\cm}[1]{\mathrm{Cm}\, #1}

\newcommand{\J}[1]{\mathcal{J}(#1)}

\newcommand{\Irr}[1]{\mathbf{I}(#1)}
\newcommand{\irr}[1]{\mathrm{I}(#1)}

\newcommand{\onto}{\rightarrow\mathrel{\mkern-14mu}\rightarrow}

\def\Ra{\Rightarrow}
\def\ra{\rightarrow}

\makeatletter
\newcommand{\doublevee}{\@doubleop{\vee}}
\newcommand{\doublewedge}{\@doubleop{\wedge}}
\newcommand{\@doubleop}[1]{%
  \DOTSB\mathop{\mathpalette\@doubleop@aux{#1}}
}

\newcommand\@doubleop@aux[2]{%
  \sbox\z@{$\m@th#1#2$}%
  \makebox[1.25\wd\z@][s]{$\m@th#1#2\hss#2$}%
}

\newcommand{\bigdoublevee}{\big@doubleop{\bigvee}}
\newcommand{\bigdoublewedge}{\big@doubleop{\bigwedge}}
\newcommand{\big@doubleop}[1]{%
  \DOTSB\mathop{\mathpalette\big@doubleop@aux{#1}}\slimits@
}

\newcommand\big@doubleop@aux[2]{%
  \sbox\z@{$\m@th#1#2$}%
  \makebox[1.35\wd\z@][s]{$\m@th#1#2\hss#2$}%
}
\makeatother

\def\up#1{{#1}\mathord{\uparrow}}
\def\dw#1{{#1}\mathord{\downarrow}}

\def\pw{\mathscr{P}}
\def\glue#1#2#3#4{#1^#2\star{}^#3\kern-0.2ex#4}
\def\HtM{\overline{M}}

\def\comm#1#2{\mathbin{[#1,#2]}}
\def\op{\bigstar}

\tikzstyle{every label}=[label distance=0pt]
\tikzstyle{bdot}[1.5]=[circle,fill,draw,thick,minimum size=#1mm,inner sep=0pt]
\tikzstyle{dot}[1.5]=[circle,draw,thick,minimum size=#1mm,inner sep=0pt]
\tikzstyle{sdot}[1.5]=[rectangle,draw,thick,minimum size=#1mm,inner sep=0pt]
\tikzstyle{rdot}[1.5]=[circle,fill=red,draw=red,thick,minimum size=#1mm,inner sep=0pt]
\tikzstyle{every edge}=[draw=black,thick]

\title[Free Hilbert algebras]{A generic construction of free algebras
in varieties of Hilbert algebras and Brouwerian semilattices} 

\author[Kowalski]{Tomasz Kowalski$^{2,3,4}$}
\author[Słomczyńska]{Katarzyna Słomczyńska$^{1}$}

\address{$^{1}$ Department of Mathematics,
University of the National Education Commission, Kraków}
\email{irena.korwin-slomczynska@uken.krakow.pl} 

\address{$^{2}$ Department of Logic, Jagiellonian University}
\email{tomasz.s.kowalski@uj.edu.pl}

\address{$^{3}$ Department of Physical and Mathematical Sciences, La Trobe
  University} 
\email{t.kowalski@latrobe.edu.au}

\address{$^{4}$ School of Historical and Philosophical Inquiry,
  The University of Queensland} 
\email{t.kowalski@uq.edu.au}

\begin{document}

\maketitle

\begin{abstract}
We give a generic construction of the $n$-generated free Hilbert algebras and
Brouwerian semilattices by a uniform
method applicable to any variety of Hilbert algebras or Brouwerian
semilattices, as long as a manageable description of subdirectly irreducible
algebras in that variety is available.
A slight modification yields analogous constructions for varieties of Hilbert
algebras and Brouwerian semilattices with zero. As examples we construct free
algebras in varieties of bounded height and of bounded width; we also find a
closed formula for the free spectrum of linear Hilbert algebras. Finally, we 
obtain a few results on structural completeness of varieties of Hilbert algebras
with zero, in particular, we give a sufficient condition for a quasi-equation to
be equivalent to an equation.  
\end{abstract}

\section{Introduction}\label{sec:intro}

Brouwerian semilattices ($\mathsf{Br}$) are among the most investigated classes
of \emph{varieties of   logic}. In the seminal article of K\"ohler~\cite{Koh81}
in particular, a construction of the free 
Brouwerian semilattice is presented, based in part on an earlier
work of de Bruijn~\cite{deB75}. It leads to a recursive formula for the number
of meet-irreducible elements, which was used to compute the exact number
of elements of the 3-generated free Brouwerian semilattice\footnote{Independently
  calculated in Krzystek~\cite{Krz77}.}. 
Of the more recent works, Bezhanishvili et al.~\cite{BBCGGJ21}
use a colouring technique to construct free algebras in \emph{nuclear Brouwerian
  semilattices}, which expand Brouwerian semilattices with a strong form of a
nucleus (an S4-like modality).

Hilbert algebras ($\mathsf{Hi}$), implicative
subreducts of Brouwerian semilattices, are less studied. A number of
observations can be found in de Bruijn~\cite{deB75} and
Urquhart~\cite{Urq74}: for free algebras, these amount to noticing that the
free Hilbert algebra generated by $X$ can be obtained from the free Brouwerian
semilattice generated by $X$ by using only the arrow for generating.
It is certainly true, but does not elucidate the structure of free Hilbert
algebras. More recently, finite Hilbert algebras (but
not the free algebra) were 
studied in Celani and Cabrer~\cite{CC05}.
All these works use techniques based either on a syntactic analysis of terms,
or on dualities and Kripke frames.

\subsection{Outline of the article}
We take a different route: a purely algebraic one, motivated by
Słomczyńska~\cite{Slo08}. Beginning from
an $n$-generated free algebra $\FV(n)$---which obviously exists and is
finite\footnote{By local finiteness of Hilbert algebras.}---in some subvariety
$\mathcal{V}$ of $\mathsf{Hi}$ 
(or of its variant $\mathsf{Ho}$ expanded by zero),
we analyse the structure of $\FV(n)$ given by its subdirectly
irreducible factors, that is, its completely meet-irreducible congruences.
The poset $\Cm{\FV(n)}$ of these congruences is isomorphic
to the poset $\Cm{\FVhat(n)}$ of completely meet-irreducible congruences
of the free algebra $\FVhat(n)$, where $\widehat{\mathcal{V}}$ is a natural
counterpart of $\mathcal{V}$ in $\var{Br}$ (or $\var{Bo}$, the zero-expanded variant).

Since the lattice of subvarieties of  $\var{Br}$ embeds in the lattice of
subvarieties of $\var{Hi}$, and the same holds for
$\var{Ho}$ and $\var{Bo}$ (see Theorem~\ref{thm:subvariety-isom}), it gives us
skeletons  (Kripke frames) for free algebras in all these varieties in a uniform way.

The frames have a multilayer structure, corresponding to the ever higher
subdirectly irreducibles (with height measured by chains of meet-irreducible
congruences). The crucial difference between varieties with zero 
and the varieties without zero occurs only at the first layer, higher up the
construction of the frame is the same. 

The next step, where the distinction between Brouwerian semilattices and Hilbert
algebras comes into play, is the choice of some antichains in the frames.
If we take all antichains, we obtain a Brouwerian semilattice, if we omit some,
we lose some meets, so we obtain a \emph{partial} Brouwerian
semilattice. Carefully selecting the right antichains we arrive at our
description of $\FV(n)$. Algorithms~1 and~2 of Section~\ref{sec:append}
show how to build
$\Cm{\FV}(n)\cong\mathbb{P}(n) = \bigcup_{i=1}^kP_i(n)$, generically, for
$\mathcal{V} = \var{Br}, \var{Hi}, \var{Bo}, \var{Ho}$.
Now, to build the free Hilbert algebra (without or with zero)
we need a function selecting the right antichains. For any generator $x$, we put 
$$
S(x) \deq \big\{(L,i)\in\mathbb{P}(n):
x\notin L \text{ and } x\in L'\text{ for all } (L',j)>(L,i)\big\},
$$
with the notation explained in Algorithm~1. Here is
a ``baby version'' of our main results, that is, of
Theorems~\ref{thm:free-Hilbert-alg} and~\ref{thm:free-Hi0-alg}.
\begin{theorem}\label{thm:baby-version}
For any $n > 0$, we have
$$
\alg{F}_{\var{Hi}}(n)\cong \biggl(\bigcup_{i=1}^n\pw(S(x_i));\ra,\emptyset\biggr)
\quad
\alg{F}_{\var{Ho}}(n)\cong
\biggl(\bigcup_{i=1}^{n}\pw(S(x_i))\cup \pw(S(0));\ra,\emptyset,P_1(n)\biggr)
$$
where $\pw$ is the power set operation, $1$, $0$ are interpreted as $\emptyset$,
$P_1(n)$ and $\ra$ is defined suitably.
\end{theorem}

To deal with subvarieties generically, we introduce appropriate restrictions:
technically they depend on which members of a variety $\mathcal{V}$ can be
``masted'' in $\mathcal{V}$ (see the remarks immediately preceding
Theorem~\ref{thm:Hi-si}). 
Having built $\FV(n)$ for some variety $\mathcal{V}$ of Hilbert algebras,
we obtain $\FVhat(n)$ for the corresponding variety
$\widehat{\mathcal{V}}$ of Brouwerian semilattices
(see Theorems~\ref{thm:free-Br-semi} and~\ref{thm:free-Br-semi-0})
by closing $\FV(n)$ under meets. So, we build free Brouwerian semilattices
by expanding free Hilbert algebras, 
rather than building free Hilbert algebras by subreducing free Brouwerian
semilattices.

\section{Preliminaries}
Our general algebraic notation and terminology is standard, based largely on
\cite{ALVI}, and \cite{Ber11}.
Below we recall notation we will very frequently use
in the article, as well as a few lesser known universal algebra notions
that will be important for us.

For an algebra $\alg{A}$ we write
$\con{\alg{A}}$ for the set of congruences on $\alg{A}$, and
$\Con{\alg{A}}$ for the lattice of congruences on $\alg{A}$.
Further, we write $\cm{\alg{A}}$ for the
set of completely meet-irreducible congruences of $\alg{A}$, and
$\Cm{\alg{A}}$ for the poset $(\cm{\alg{A}};\subseteq)$ of completely
meet-irreducible congruences ordered by inclusion. 
We write $\subsetneq$ for proper inclusion and $\subseteq$ for inclusion; the
distinction will be crucial at several places.
Each congruence $\mu\in\cm{\alg{A}}$ has a unique cover, which we will denote $\mu^+$.
Dually, each completely join-irreducible congruence $\alpha$ has a unique
subcover, which we will denote $\alpha_-$. Intervals in congruence lattices will
be denoted by $\mathrm{I}[\alpha,\beta]$ to distinguish them notationally from
the commutator $[\alpha,\beta]$ that we will occasionally use.

For a poset $(P;\leq)$, we write $\mathrm{Up}(P;\leq)$ for the set of
upsets of $(P;\leq)$. Very often the order relation will be natural and clear
from context, in such cases we will write $P$ and $\mathrm{Up}(P)$.
We write $\up{X}$, $\dw{X}$ for respectively the upset and
the downset generated by $X\subseteq P$, and $X^\complement$ for the set complement
of $X$. 
 
An algebra $\alg{A}$ is called \emph{pointed} if $\alg{A}$ has a term-definable
constant. In this article we will denote it uniformly by $1$.
Typically, but not necessarily,
$1$ is a constant in the signature. A class of algebras
$\mathcal{K}$ is pointed if every $\alg{A}\in\mathcal{K}$ is pointed
with $1$ defined by the same term. The set of all terms in
$n$ variables over a signature $\Sigma$, will be denoted by
$\mathrm{Term}_\Sigma(n)$, with $\Sigma$ 
written in some convenient form clear from context, for example,
$\mathrm{Term}_{\{\ra,1\}}(n)$, $\mathrm{Term}_{\{\ra,\wedge,1\}}(n)$,
$\mathrm{Term}_{\{\ra,0,1\}}(n)$, etc.

A pointed algebra $\alg{A}$ is
\emph{congruence orderable with respect to $1$} (or simply
\emph{orderable}) if the preordering relation $\leq$, defined
by $x\leq y$ if{f} $\theta(y,1)\subseteq\theta(x,1)$, is an ordering.
In any orderable algebra, $1$ is the top element with respect to the congruence
ordering. A class 
$\mathcal{K}$ is orderable if every $\alg{A}\in\mathcal{K}$ is.
The next result (cf. Lemma~2.1  of~\cite{ISW09}) characterises
subdirectly irreducible algebras in orderable varieties.

\begin{lemma}\label{lem:1-orderable-si}
A variety $\mathcal{V}$ is orderable if and only if for any subdirectly
irreducible $\alg{A}\in\mathcal{V}$ with monolith $\mu$ we have
$|1/\mu| = 2$ and $|a/\mu| = 1$ for any $a\notin 1/\mu$.
\end{lemma}

Thus, in orderable varieties
subdirectly irreducible algebras all have a unique
subcover of $1$, that is, a unique coatom.
Having a unique does not 
force subdirect irreducibility, in orderable varieties (pointed
distributive lattices furnish a counterexample), unless we add one more
condition: congruence 1-regularity.

A pointed algebra $\alg{A}$ is  
\emph{congruence 1-regular} (or simply 
\emph{1-regular}) if $1/\alpha = 1/\beta$ implies $\alpha = \beta$,
for any $\alpha,\beta\in\Con{\alg{A}}$. A class $\mathcal{K}$ is 1-regular
if every $\alg{A}\in\mathcal{K}$ is. In 1-regular varieties, congruence
blocks of $1$ play an important role, 
and they are often called \emph{filters} (or in a dual version, \emph{ideals};
see~\cite{GU84} for the beginnings of the theory of ideals in universal algebra,
later developed in~\cite{Urs94, AU96, AU97, AU97a, Urs00}).
We will return to filters in the
specific context of Hilbert algebras and Brouwerian semilattices later.

If $\alg{A}$ is both orderable and 1-regular,
then $\alg{A}$ is called \emph{Fregean} and the same applies to classes of
algebras. The theory of Fregean algebras began with~\cite{BKP84}, and was
developed in~\cite{ISW09} and~\cite{Slo12, Slo22}. The following
characterisation of subdirectly irreducibles in Fregean varieties
is Proposition~3.1 of~\cite{Slo12}. 

\begin{lemma}\label{lem:1-fregean-si}
Let $\mathcal{V}$ be a Fregean variety, and let $\alg{A}\in\mathcal{V}$. Then
$\alg{A}$ is subdirectly irreducible if and only if $\alg{A}$ has a unique coatom. 
\end{lemma}

The recognition of the importance of the unique-coatom property predates
the notions of orderability and Fregeanity by decades: it
was first identified for Heyting algebras. Following the
tradition, we will denote the unique coatom by $\op$.

Let $\alg{A}$ be an arbitrary algebra. The map
$\HtM\colon \Con\,\alg{A} \to \mathrm{Up}(\Cm{\alg{A}})$ 
given for any $\alpha\in\Con\,\alg{A}$ by
$\HtM(\alpha) \deq \{\mu\in\cm{\alg{A}}: \alpha\subseteq\mu\}$ 
is injective, by Birkhoff Theorem. For orderable algebras there is a better
result.

\begin{theorem}\label{thm:Heyting-from-cm}
Let $\alg{A}$ be an orderable algebra.
The map $M\colon \alg{A}\to \mathrm{Up}(\Cm{\alg{A}})$,
given by
$M(a)\deq \HtM(\theta(1,a)) = \{\mu\in\Cm{\alg{A}}: (1,a)\in\mu\}$
is injective, and
$$
(\mathrm{Up}(\Cm{\alg{A}}); \cup,\cap, \ra,\emptyset,\cm{\alg{A}})
$$
is a Heyting algebra, with $U\ra W$ defined as
$\dw{(U\smallsetminus W)}^\complement$
where ${}^\complement$ stands for the set complement.
\end{theorem}

A pointed algebra $\alg{A}$ is \emph{subtractive} if 
there is a binary term $s$ in the language of $\alg{A}$ such that
$\alg{A}\models s(x,x) = 1$ and $\alg{A}\models s(1,x) = x$.
A class $\mathcal{K}$ is subtractive if every $\alg{A}\in\mathcal{K}$ is
subtractive with the same term $s$.
The theory of subtractive varieties, due overwhelmingly to Aglian\`o and Ursini,
cf.~\cite{Urs94, AU96, AU97, AU97a, Urs00}, is built upon the observation that
many important properties of classical 
algebras such as groups or rings depend only on the existence of a subtractive
term. Although we do not assume
proficiency in this theory, some familiarity with it will help the reader
greatly.

There are three paradigmatic Fregean varieties:
\emph{equivalential algebras} ($\var{Eq}$), \emph{Brouwerian semilattices} 
($\var{Br}$) and \emph{Hilbert algebras} ($\var{Hi}$). Indeed, by Theorem~3.8
in~\cite{ISW09},  
every CP Fregean variety is term-equivalent to a variety of equivalential
algebras with additional operations. By Corollary~4.1 in~\cite{ISW09}, every
arithmetical Fregean variety is term-equivalent to a variety of Brouwerian
semilattices with \emph{compatible} operations; by Theorem~3.4 of~\cite{Agl01}
every Fregean variety with equationally definable principal congruences (EDPC)
is term-equivalent to a variety of Hilbert algebras with compatible operations.
Since $\var{Eq}$, $\var{Br}$, and $\var{Hi}$ play such prominent roles among
Fregean varieties it is natural to investigate free algebras in these varieties.
Free equivalential algebras were studied in~\cite{Slo08}, but we will leave
these aside, and focus on Hilbert algebras and Brouwerian semilattices.
Free Brouwerian semilattices were studied in~\cite{Koh81}, and 
certain constructions leading towards free Hilbert algebras were given
in~\cite{Die66} and in~\cite{Urq74}.

\section{Hilbert algebras and Brouwerian semilattices}\label{sec:Hi}

Hilbert algebras are implication subreducts of Heyting algebras. More
formally we can define an algebra $\alg{A} = (A; \ra,1)$ to be a Hilbert
algebra if it satisfies the following identities:
\begin{enumerate}[label={(H\arabic*)}]
\item $x\ra x = 1$,
\item $1\ra x = x$,
\item $x\ra (y\ra z) = (x\ra y)\ra(x\ra z)$,    
\item $(x\ra y)\ra ((y\ra x)\ra x) = (y\ra x)\ra ((x\ra y)\ra y)$.
\end{enumerate}
Here are a few useful consequences of the defining identities:
\begin{enumerate}[label={(H\arabic*)}]
\setcounter{enumi}{4}  
\item $x\ra (y\ra x) = 1$,  
\item $x\ra (y\ra z) = y\ra (x\ra z)$,
\item $x\ra y = 1 \ \mathbin{\&}\  y\ra x = 1 \ \Ra\  x = y$.
\item $x\ra y = 1 \implies x\ra(y\ra z) = x\ra z$.
\end{enumerate}
The identities (1) and (2) show that $x\ra y$ is a subtractive term.
By (3), for any $a\in A$, the unary
operation $h_a(x) = a\ra x$ is a homomorphism. 
It is immediate that the relation $\leq$ defined by
$x\leq y \iff x\ra y = 1$ is an order on any Hilbert algebra; it is easy to
show that Hilbert algebras are congruence 1-regular and orderable
(with the congruence ordering coinciding with $\leq$), hence Fregean.
Moreover $\var{Hi}$ has equationally definable principal congruences
(EDPC), which implies congruence-distributivity (CD), and is locally finite
by a well-known result of Diego~\cite{Die66}. 
Furthermore, $\var{Hi}$ is precisely the class of subreducts
of \emph{Brouwerian semilattices}, that is, algebras
$\alg{B} = (B;\wedge,\ra,1)$ such that $(B;\wedge,1)$ is a semilattice with
the largest element $1$, and $\ra$ is the residual of $\wedge$, that is,
$$
x\wedge y\leq z \iff y\leq x\ra z
$$
holds for any $x,y,z\in B$. As residuation equivalences are themselves equivalent
to (some) identities, Brouwerian semilattices are a variety, which we will
denote by $\var{Br}$. By an argument paralleling~\cite{Die66}, to be found e.g.,
in~\cite{Koh81}, $\var{Br}$ is locally finite. An equational base for $\var{Br}$
is given by   
\begin{enumerate}[label={(B\arabic*)}]
\item identities defining a meet semilattice with top element $1$,  
\item $x\ra x = 1$, 
\item $x\wedge(x\ra y) = x\wedge y$,
\item $(x\wedge y)\ra z = x\ra(y\ra z)$.
\end{enumerate}
A useful consequence of these is:
\begin{enumerate}[label={(B\arabic*)}]
\setcounter{enumi}{4}
\item $x\ra (y\wedge z) = (x\ra y)\wedge(x\ra z)$.
\end{enumerate}  
Analogously, $\var{Hi}$ is precisely the class of $\{\ra,1\}$-subreducts of Heyting
algebras, and an important subvariety of Hilbert algebras, the variety of
\emph{Tarski algebras}, defined by the identity
$(x\ra y)\ra x = x$, is the class of 
$\{\ra,1\}$-subreducts of Boolean algebras. Since reducts
of Brouwerian semilattices will appear quite often in the article,
we introduce a notation: for a Brouwerian semilattice $\alg{B}$ we write
$\alg{B}^\ra$ for the $\{\ra,1\}$-reduct of $\alg{B}$ and extend this notation
to classes. 

For use in Section~\ref{sec:free-Hi}, we recall here a
characterisation of subdirectly irreducible Hilbert algebras.
For any $\alg{A}\in\var{Hi}$ we write 
$\alg{A}^\oplus$ for the algebra constructed as follows.
Rename the top element of $A$ to $\op$, define
$A^\oplus = A\uplus \{1\}$, extend the ordering on $A$ to $A^\oplus$
putting $x \leq^{\alg{A}^\oplus} 1$ for every $x\in A^\oplus$; then put
$$
x \ra^{\alg{A}^\oplus} y =\begin{cases}
  1 & \text{ if } x\leq^{\alg{A}^\oplus} y,\\
  y & \text{ if } x = 1,\\
  x\ra^{\alg{A}} y & \text{ otherwise.}
\end{cases}
$$
The construction of $\alg{A}^\oplus$, often referred to as the ``masting'' of
$\alg{A}$, can be alternatively described as the \emph{ordinal
  sum} (see e.g.,~\cite{BF00})
$\alg{A}\oplus\alg{2}$, where 
$\alg{2} = (\{0,1\}; \ra, 1)$ with $x\ra y = 0 \iff x = 1 \mathbin{\&} y = 0$.

\begin{theorem}\label{thm:Hi-si}
Let $\mathcal{V}$ be a variety of Hilbert algebras or of Brouwerian semilattices,
and let $\alg{A}\in\mathcal{V}$.  
Then, $\alg{A}^\oplus$ is subdirectly
irreducible (possibly not in $\mathcal{V}$),
and every subdirectly irreducible $\mathcal{V}$-algebra is of that form
for some $\alg{A}\in\mathcal{V}$.
Moreover, $\alg{A}\cong\alg{A}^\oplus\smallsetminus\{\op\}\cong 
\alg{A}^\oplus/\mu$, where $\mu$ is the monolith;
in particular, $1/\mu = \{\op, 1\}$.
Furthermore $A^\oplus\smallsetminus\{\op\}$ is a subuniverse of
$\alg{A}^\oplus$. 
\end{theorem}  

The `furthermore' part of the above theorem makes the  
standard argument for local finiteness work, showing that all subvarieties
of $\var{Hi}$ and of $\var{Br}$ are locally finite. 

We come back to filters now. A \emph{filter} in a Hilbert algebra (or
a Brouwerian semilattice) $\alg{A}$ is a subset $F\subseteq A$ such that
$1\in F$,  and $a,a\ra b\in F$ imply $b\in F$ for any $a,b\in A$.
If $F$ is a filter, then $F = 1/\theta(F)$, where
$\theta(F)$ is the congruence generated by $\{(f,1): f\in F\}$. Conversely,
for each congruence $\vartheta\in\con{\alg{A}}$, the block
$1/\vartheta$ is a filter. These correspondences are mutually inverse maps
establishing an 
isomorphism between the lattice of filters and $\Con{\alg{A}}$. Precisely the
same holds for Brouwerian semilattices, where additionally filters in the above
sense coincide with semilattice filters. Furthermore, the following analogue of
the Prime Filter Theorem holds.

\begin{theorem}\label{thm:PrimeFilter}
Let $\alg{A}$ be a Hilbert algebra or a Brouwerian semilattice.
For any $a\in A$ and any $X\subseteq A$, if $\up{a}\cap X = \emptyset$, then
there is a $\mu\in\Cm{\alg{A}}$ such that $a\in 1/\mu$ and
$X\cap 1/\mu = \emptyset$.
\end{theorem}  

The lattices of subvarieties $\Lambda(\var{Hi})$ and $\Lambda(\var{Br})$,
are also intimately connected: in fact the latter embeds in the former.
To show this, we first recall a result of
Celani and Jansana~\cite[Thms.~6.7, 6.8, Prop.~7.2]{CJ12}. 

\begin{lemma}\label{lem:min-ext}
For every finite $\alg{A}\in\var{Hi}$ there exists a 
$\alg{B}\in\var{Br}$, unique up to isomorphism,
such that $\alg{A}\hookrightarrow \alg{B}^\ra$, and
for any $\alg{C}\in\var{Br}$ and any
$\varphi\colon \alg{A}\hookrightarrow \alg{C}^\ra$, there is a unique
$\overline{\varphi}\colon \alg{B}\hookrightarrow \alg{C}$ such that
$\overline{\varphi}|_{\alg{A}} = \varphi$.
Moreover, $\Con{\alg{A}} \cong \Con{\alg{B}}$. 
\end{lemma}

For any Hilbert algebra $\alg{A}$, we will write $\langle \alg{A}\rangle$
to denote the Brouwerian semilattice $\alg{B}$ with the properties from the
above lemma. We will refer to $\langle \alg{A}\rangle$
as the \emph{free extension} of $\alg{A}$. Note that for any
$\alg{A}\in\var{Hi}$ and any $\alg{B}\in\var{Br}$ we have
$\alg{A}\in S(\alg{B}^\ra)$ if and only if
$\langle \alg{A}\rangle\in S(\alg{B})$.

Next, since in $\var{Hi}$ and $\var{Br}$ any identity $s=t$ is
equivalent to a pair of 
identities $s\ra t = 1$ and $t\ra s=1$, we can consider only identities of the form
$t=1$. For any $t\in\mathrm{Term}_{\{\wedge,\ra,1\}}$ the identity $t = 1$ is 
equivalent over $\var{Br}$ to a finite set of identities
$\{s = 1: s\in H(t)\}$, where $H(t)\subseteq \mathrm{Term}_{\{\ra,1\}}$ is obtained
by a recursive rewriting procedure based on the following:
\begin{enumerate}
\item If $t = t_1\wedge t_2$, then $t=1\iff t_1=1 \mathbin{\&} t_2 = 1$. 
\item If $t = s\ra (t_1\wedge t_2)$, then
$t=1\iff s\ra t_1=1 \mathbin{\&} s\ra t_2 = 1$.  
\item If $t = (t_1\wedge t_2)\ra s$, then
$t=1\iff t_1\ra(t_2\ra s)=1$. 
\end{enumerate}
Obviously, for any $\mathcal{V}\subseteq\var{Br}$ and any term
$t\in \mathrm{Term}_{\{\wedge,\ra,1\}}$ we have
$\mathcal{V}\models t=1$ if and only if $\mathcal{V}\models s=1$ for all $s\in H(t)$.

Hence, any variety $\mathcal{V}\in\Lambda(\var{Br})$ has an equational base relative to
$\var{Br}$ in the signature of $\var{Hi}$. Denoting this base by
$\mathrm{Eq}^\ra(\mathcal{K})$, we define
$f\colon \Lambda(\var{Br})\to \Lambda(\var{Hi})$
putting $f(\mathcal{V}) = \var{Hi}\cap
\mathrm{Mod}(\mathrm{Eq}^\ra(\mathcal{V}))$.

\begin{lemma}\label{lem:f-charact}
Let $f\colon \Lambda(\var{Br})\to \Lambda(\var{Hi})$ be the map defined above.
Then $f(\mathcal{V}) = S(\mathcal{V}^\ra)$, for any $\mathcal{V}\in\var{Br}$.   
\end{lemma}

\begin{proof}
Note the following: (i) reducts commute with direct products, (ii)
$\var{Hi}$ has CEP, (iii) $\alg{B}^\ra/\varphi = (\alg{B}/\varphi)^\ra$ for any
$\varphi\in\Con{\alg{B}}$, by Lemma~\ref{lem:min-ext}. Hence
$$
f(\mathcal{V}) = HSP(\mathcal{V}^\ra) = HS(\mathcal{V}^\ra) = SH(\mathcal{V}^\ra)
= S(\mathcal{V}^\ra)
$$
where the first equality is definitional, the second follows by (i), the third
by (ii) and the fourth by (iii).
\end{proof}  

In particular, $S(\mathcal{V}^\ra)$ is a variety. In the case of Hilbert
algebras this holds trivially since $S(\mathcal{V}^\ra)$ is a quasivariety and
$\var{Hi}$ is hereditarily structurally complete, but we will apply
Lemma~\ref{lem:f-charact} to cases where this is no longer true
(see Section~\ref{sec:zero}).

\begin{theorem}\label{thm:subvariety-isom}
The lattice of subvarieties of\/ $\var{Br}$ embeds
into the lattice of subvarieties of\/ $\var{Hi}$.
\end{theorem}  

\begin{proof}
Note that $\alg{A}\in f(\mathcal{V})\iff \langle\alg{A}\rangle\in \mathcal{V}$. 
We will show that $f$ is injective and preserves lattice operations.
For injectivity, let $\mathcal{V},\mathcal{W}\in\Lambda(\var{Br})$ be distinct.
Without loss there is an algebra $\alg{A}\in \mathcal{V}\setminus\mathcal{W}$. Then,
$\alg{A}^{\ra} \in \mathcal{V}^\ra\setminus\mathcal{W}^\ra$.

Next, consider $\mathcal{V}\cap\mathcal{W}$. Then
$\alg{A}\in f(\mathcal{V}\cap\mathcal{W}) \iff
\langle\alg{A}\rangle \in \mathcal{V}\cap\mathcal{W}
\iff \langle\alg{A}\rangle \in \mathcal{V} \text{ and }
\langle\alg{A}\rangle \in \mathcal{W}
\iff \alg{A}\in f(\mathcal{V})\cap f(\mathcal{W})$.

Finally, consider $\mathcal{V}\vee\mathcal{W}$. By monotonicity we have
$f(\mathcal{V})\vee f(\mathcal{W})\subseteq f(\mathcal{V}\vee\mathcal{W})$.
For converse, it suffices to show that all subdirectly irreducible
algebras from
$(f(\mathcal{V}\vee\mathcal{W}))_{SI}\subseteq
f(\mathcal{V})\vee f(\mathcal{W})$. Let $\alg{A}\in
f(\mathcal{V}\vee\mathcal{W})_{SI}$; then
$\langle\alg{A}\rangle \in \mathcal{V}\vee\mathcal{W}$.
Since $\Con{\alg{A}}\cong \Con{\langle\alg{A}\rangle}$
the algebra $\langle\alg{A}\rangle$ is also subdirectly irreducible,
so $\langle\alg{A}\rangle \in (\mathcal{V}\vee\mathcal{W})_{SI}
= (\mathcal{V})_{SI}\cup(\mathcal{W})_{SI}$ by CD and J\'onsson's Lemma.
Hence $\langle\alg{A}\rangle \in \mathcal{V}_{SI}$
or $\langle\alg{A}\rangle \in \mathcal{W}_{SI}$, and thus, without loss of generality
$\alg{A}\in f(\mathcal{V})\subseteq f(\mathcal{V})\vee f(\mathcal{W})$.
\end{proof}  

\subsection{Pure Hilbert algebras}\label{sec:pure}

For any poset $P$, define an algebra $\algov{P} = (\ov{P}; \ra, 1)$ where
$\ov{P} = P\uplus \{1\}$ and
$$
x\ra y =
\begin{cases} 
1 & \text{ if } x\leq y\\
y & \text{ otherwise} 
\end{cases}
$$
for all $x,y\in \ov{P}$. It is easy to see that for any poset $P$, the algebra
$\algov{P}$ is a Hilbert algebra. These algebras are called \emph{pure Hilbert
  algebras} in~\cite{BB06}, where it is shown (see~\cite[Theorems~2.2
and~2.3]{BB06}) that  
the subdirectly irreducible members of the variety generated by pure Hilbert
algebras are themselves pure, and the variety is finitely based. No 
equational base is given explicitly, so we give one here.

\begin{theorem}\label{thm:pure-base}
The variety generated by pure Hilbert algebras is based, relative to Hilbert
algebras, by the identity
\begin{equation}\label{eq:P}\tag{P}
\bigl((x\ra y)\ra z\bigr)\ra \Bigl(\bigl(((x\ra y)\ra y)\ra z\bigr)\ra z\Bigr) = 1.
\end{equation}  
\end{theorem}

\begin{proof}
It is easy to see that~\eqref{eq:P} holds in every pure Hilbert
algebra. Conversely, let $\alg{A}$ be a subdirectly irreducible Hilbert algebra
which is not pure.
Then, there are elements $a, b\in A$ such
that $b < a\ra b < 1$. Therefore $b < a\ra b \leq \op$ and
$(a\ra b)\ra b \leq \op$. Hence
$\bigl((a\ra b)\ra \op \bigr)\ra \Bigl(\bigl(((a\ra b)\ra b)\ra \op \bigr)\ra
\op\Bigr) = 1 \ra  (1\ra \op ) = \op$, falsifying~\eqref{eq:P}.
\end{proof}

Pure Hilbert algebras witness the fact that the embedding
of $\Lambda(\var{Br})$ into $\Lambda(\var{Hi})$ is not a surjection.
For consider a subdirectly irreducible pure Hilbert algebra $\alg{A}$ which is not
linearly ordered, and let $\alg{B}$ be a subdirectly irreducible Brouwerian
semilattice such that $\alg{A}$ is a subreduct of $\alg{B}$.
Then there are $a,b\in A$ such that $a$ and $b$ are incomparable, so
we have $b< a\ra (a\wedge b) = a\ra b < 1$ in $\alg{B}$. Hence~\eqref{eq:P} fails in
$\alg{B}$.

\section{Congruence orderable subtractive varieties}\label{sec:orderable}

To place our results in a broader context, we give
some background on varieties which are congruence orderable 
and subtractive. These two properties are in general independent
(groups are subtractive but not orderable, and pointed sets are orderable but not
subtractive) but taken together they imply Fregeanity (see for example,
Theorem~2.1 of~\cite{Agl01}, modulo correcting an obvious typo in the proof: the
first displayed formula there should read $s(x,y) = 0 = s(y,x) \iff x = y$).

We will state something slightly stronger. The theorem below can be proved in
exactly the same way as Proposition~3.6 of~\cite{Slo12}, by noticing only that
the assumption of Fregeanity in that proposition is too strong: the proof uses only
orderability, and Fregeanity follows. 

\begin{theorem}\label{thm:1-orderable-subtr}
If a variety $\mathcal{V}$ is orderable and has a subtractive term $s$, then
for every $\alg{A}\in\mathcal{V}$ and every $a,b\in A$ we have
$\theta(a,b) = \theta(1,s(a,b))\vee\theta(1,s(b,a))$. In particular,
$\mathcal{V}$ is Fregean.
\end{theorem}  

The next result gives sufficient conditions for an orderable subtractive
variety $\mathcal{V}$ to be congruence permutable (CP) and congruence
distributive (CD).

\begin{theorem}\label{thm:subtr-order-CD-CP} 
Let $\mathcal{V}$ be an orderable variety with a subtractive term $s(x,y)$.
\begin{enumerate}
\item If\/ $\mathcal{V}\models s(x,1) = x$, then $\mathcal{V}$ is CP.
\item If\/ $\mathcal{V}\models s(x,1) = 1$, then $\mathcal{V}$ is CD.
\end{enumerate}
\end{theorem}

\begin{proof}
For (1), note that by the assumptions
$\mathcal{V}\models s(x,1) = x = s(1,x)$ and $\mathcal{V}\models s(x,x) = 1$.
Then, by Theorem~3.8 of~\cite{ISW09}, $\mathcal{V}$ is CP.

For (2) we need to work a little harder. Since $\mathcal{V}$ is 1-regular, it is
congruence modular (CM), the standard commutator theory of~\cite{COMM} 
applies. Recall that
the commutator $\comm{\alpha}{\beta}$ of congruences $\alpha$ and $\beta$ is the
smallest congruence such that for any $k$-ary term $t$ (with $k\geq 2$),
any $(a,b)\in \alpha$ and any $c_1,\dots,c_{k-1}, d_1,\dots,d_{k-1}$ such that
$(c_i,d_i)\in\beta$ for all $i\in\{1,\dots,k-1\}$ we have
$$
t(a,\overline{c})\comm{\alpha}{\beta}t(a,\overline{d})
\iff
t(b,\overline{c})\comm{\alpha}{\beta}t(b,\overline{d}).
$$
We claim that
$$
\comm{\theta(1,a)}{\theta(1,a)} = \theta(1,a)
$$
holds for any $\alg{A}\in\mathcal{V}$ and any $a\in A$.
The left-to-right inclusion is clear, since $\comm{\alpha}{\beta}\leq
\alpha\cap\beta$. To see the other inclusion,
note that by the assumptions we have $s(a,1) = s(a,a)$ and so
$s(a,1)\comm{\theta(1,a)}{\theta(1,a)} s(a,a)$. By definition
of commutator and the assumptions, then
$1 = s(1,1)\comm{\theta(1,a)}{\theta(1,a)} s(1,a) = a$.

Now, take an arbitrary $\alpha\in \alg{Con}(\alg{A})$. By regularity of 
$\mathcal{V}$ we get that $\alpha = \bigvee_{a\in 1/\alpha}\theta(1,a)$. 
Using monotonicity of commutator, for any $a\in 1/\alpha$ we obtain
$$
\theta(1,a) = \comm{\theta(1,a)}{\theta(1,a)} \subseteq \comm{\alpha}{\alpha}.
$$
Hence $\alpha = \comm{\alpha}{\alpha}$. By Corollary~2.5 of~\cite{ISW09},
Fregean varieties satisfy the commutator identity C1 of~\cite{COMM}, that is,
$\comm{\varphi}{\psi} =
(\varphi\cap\comm{\psi}{\psi})\vee(\psi\cap\comm{\varphi}{\varphi})$.
It follows that $\comm{\varphi}{\psi} = \varphi\cap\psi$ for any
congruences $\varphi$, $\psi$, and this for CM varieties is equivalent to
CD.
\end{proof}  

The apparent symmetry between condition (1) and (2) of
Theorem~\ref{thm:subtr-order-CD-CP} is somewhat misleading. All pointed CP varieties
are subtractive, via $s(x,y)\deq p(y,x,1)$ where $p$ is the Mal'cev term,
and, if they are also orderable, then by Theorem~3.8 of~\cite{ISW09} they have some
subtractive term satisfying (1).  On the other hand, there are
orderable CD varieties which are not subtractive, for example upper-bounded
distributive lattices. However, if an orderable CD variety with a subtractive
term $s(x,y)$ has a majority term $M(x,y,z)$, then it has a subtractive term
$s'(x,y) = M(s(x,y),y,1)$ satisfying (2). 

The next theorem lays foundations for a different proof of Theorem~3.4
of~\cite{Agl01}, which is our Corollary~\ref{cor:agliano} below.   

\begin{theorem}\label{thm:oprema}
Let $\mathcal{V}$ be an orderable variety, and let $s$ be a binary term in its
language. The following are equivalent.
\begin{enumerate}
\item $a\leq b \iff s(a,b) = 1$ holds for all $\alg{A}\in\mathcal{V}$ and all
  $a,b\in A$, with $\leq$ being the natural congruence ordering.
\item $s$ is a subtractive term for $\mathcal{V}$, and
for every $\alg{A}\in\mathcal{V}_{SI}$ and all $a,b\in A$ we have  
$s(a,b) = \op \implies a =1$. 
\item $(A; s, 1)$ is a Hilbert algebra and
$\mathbf{Con}(A; s,1) = \Con{\alg{A}}$, 
for any $\alg{A}\in\mathcal{V}$. 
\end{enumerate}  
\end{theorem}

\begin{proof}
(1) $\Ra$ (2). We first prove that $s$ of (1) is a subtractive
term. Clearly, $s(x,x) = 1$ holds, by reflexivity of $\leq$. To prove that
$s(1,x) = x$ holds, pick any $\alg{A}\in\mathcal{V}$ and any $a\in A$.
By orderability, it suffices to show that $\theta(1,a) = \theta(1, s(1,a))$.
Note that for any $\alpha\in\Con{\alg{A}}$ we have that
$(1,a)\in\alpha$ implies $s(1,a)\mathbin{\alpha} s(1,1) = 1$, that is,
$(1, s(1,a))\in\alpha$. Now suppose $(1, s(1,a))\in\alpha$. Then,
$1/\alpha = s(1/\alpha,a/\alpha)$ and by (1) we obtain $1/\alpha\leq a/\alpha$
and so $1/\alpha = a/\alpha$. Hence $(1,a)\in\alpha$. We conclude that
$(1,a)\in\alpha$ if and only if $(1, s(1,a))\in\alpha$, for any
$\alpha\in\Con{\alg{A}}$, and therefore
$\theta(1,a) = \theta(1, s(1,a))$ as required.

For the second part of (2), let $\alg{A}\in\mathcal{V}_{SI}$ and let
$1/\mu = \{1,\op\}$, where $\mu$ is the monolith of $\alg{A}$.
To get a contradiction, suppose $s(a,b) = \op$ but $a\neq 1$. 
Then $(1,\op)\in \theta(1,a)$ and $(s(a,b), s(1,b))\in\theta(1,a)$,
so $(1, s(1,b))\in\theta(1,a)$ since $s(a,b) = \op$. But
$s(1,b) = b$ so $(1,b)\in\theta(1,a)$ and therefore $b\geq a$, which by (1)
gives $s(a,b) = 1$, contradicting the supposition.

(2) $\Ra$ (3). Take $\alg{A}\in\mathcal{V}$ and recall the
injective map $M\colon \alg{A}\to\mathrm{Up}(\Cm{\alg{A}})$,
given by $M(a) = \{\mu\in\Cm{\alg{A}}: (1,a)\in\mu\}$. Clearly,
$(\mathrm{Up}(\Cm{\alg{A}}); \ra, \cm{\alg{A}})$ 
with $U\ra W\deq \dw{(U\smallsetminus W)}^\complement$ is a Hilbert algebra.
We will show that $M(s(a,b)) = M(a)\ra M(b)$ for any $a,b\in A$, and hence
$(A; s, 1)$ is a Hilbert algebra as it embeds into one. 

Take $\mu\in M(s(a,b))$ and suppose $\mu\in \dw{(M(a)\smallsetminus M(b))}$.
Then, for some $\nu\geq\mu$ we have $\nu\in M(a)\smallsetminus M(b)$, that is,
$(1,a)\in \nu$ and $(1,b)\notin\nu$. But $(1, s(a,b))\in\nu$, so
we get $b = s(1,b) \mathbin{\nu} s(a,b) \mathbin{\nu} 1$, which yields
$(1,b)\in\nu$ giving a contradiction. Hence, $M(s(a,b))\subseteq M(a)\ra M(b)$.

For the other inclusion, suppose $\mu\in M(a)\ra M(b)$ but
$\mu\notin M(s(a,b))$. Then $(1,s(a,b))\notin\mu$ and we can find some $\nu\geq\mu$
such that $(1,s(a,b))\in\nu^+\smallsetminus\nu$. Hence, $s(a/\nu,b/\nu) = \op/\nu$, and by
(2) we obtain $a/\nu = 1/\nu$.  Moreover, $1/\nu \neq b/\nu = \op/\nu$ as $s$ is a
subtractive term. Hence, $(1,a)\in\nu$ and $(1,b)\notin\nu$, which implies
$\nu\in M(a)\smallsetminus M(b)$. It follows that
$\mu\in\dw{(M(a)\smallsetminus M(b))}$
contradicting the supposition $\mu\in M(a)\ra M(b)$.

It remains to show that $\mathbf{Con}(A; s,1) = \Con{\alg{A}}$. Clearly, every
congruence of $\alg{A}$ is a congruence of $(A; s,1)$ so we only need to prove
$\mathrm{Con}(A;s,1)\subseteq \con{\alg{A}}$.

We begin by showing that $\theta^{(A;s,1)}(1,a) = \theta(1,a)$
for every $a\in A$ and every $\alg{A}\in\mathcal{V}$. 
Obviously, it suffices to show $\theta^{(A;s,1)}(1,a)\geq \theta(1,a)$, so
take $a\in A$ and $(1,b)\in\theta(1,a)$. Seeking a contradiction, suppose
$(1,b)\notin\theta^{(A;s,1)}(1,a)$. Then $s(a,b)\neq 1$, since $(A;s,1)$ is a
Hilbert algebra. So there exists $\mu\in\cm{\alg{A}}$ such that
$(1,s(a,b))\in\mu^+\smallsetminus\mu$. Then, $s(a/\mu, b/\mu) = \op/\mu$ and
therefore, by (2), we get $a/\mu = 1/\mu$, that is $(1,a)\in\mu$. Hence
$(1,b)\in \theta(1,a)\subseteq\mu$ and so
$s(a/\mu,b/\mu) = s(1/\mu,1/\mu) = 1/\mu$ which is a contradiction.

Now, take an arbitrary $\varphi\in\mathrm{Con}(A;s,1)$. Using the equality
just proved for principal congruences, we obtain
$\varphi = \bigvee_{a\in 1/\varphi}\theta^{(A;s,1)}(1,a) =
\bigvee_{a\in 1/\varphi}\theta(1,a)$,  and so $\varphi\in\con{\alg{A}}$.
Hence, $\mathrm{Con}(A;s,1)\subseteq \con{\alg{A}}$ as required.

(3) $\Ra$ (1). Let $\alg{A}\in\mathcal{V}$ and $a,b\in A$. Then $a\leq b$ if and
only if $(1,b)\in\theta(1,a)$ if and only if $(1,b)\in \theta^{(A;s,1)}(1,a)$
if and only if $s(a,b) = 1$. The first equivalence follows by 
orderability of $\alg{A}$, the second and third follow by (3).
\end{proof}  

\begin{lemma}\label{lem:unique}
Let $\mathcal{V}$ be an orderable variety. If in the
language of $\mathcal{V}$ there is a binary term $s$
satisfying the equivalent conditions
of Theorem~\ref{thm:oprema}, then $s$ is unique.
\end{lemma}  

\begin{proof}
Let $s$ and $r$ be binary terms such that
$x\leq y$ if and only if $s(x,y) = 1$ if and only if $r(x,y) = 1$
hold in $\mathcal{V}$. We will show that $\mathcal{V}\models s = r$. 
To get a contradiction, suppose for some $\alg{A}\in\mathcal{V}$ and some
$a,b\in A$ we have $s(a,b)\neq r(a,b)$. Then there exist some $\mu\in
\Cm{\alg{A}}$ such that $(s(a,b), r(a,b))\in\mu^+\smallsetminus\mu$. Then
$\alg{A}/\mu$ is subdirectly irreducible, and
$\{s(a/\mu,b/\mu), r(a/\mu,b/\mu)\} = \{1/\mu, \op^{\alg{A}/\mu}\}$. 
Without loss of generality, $s(a/\mu,b/\mu) = 1/\mu$ and
$r(a/\mu,b/\mu) = \op^{\alg{A}/\mu}$. Then, by Theorem~\ref{thm:oprema}(2)
applied to $r$ we get $a/\mu = 1/\mu$ and $b/\mu = \op^{\alg{A}/\mu}$.
But then $s(1/\mu,\op^{\alg{A}/\mu}) = 1/\mu$ contradicting the fact that
$s$ is a subtractive term on $\alg{A}/\mu$.
\end{proof}

\begin{cor}\label{cor:agliano}
The following are equivalent for a variety $\mathcal{V}$.
\begin{enumerate}  
\item $\mathcal{V}$ is orderable, subtractive and has EDPC.
\item $\mathcal{V}$ has a unique binary term $s$ such that
$(A;s,1)$ is a Hilbert algebra and
$\con{\alg{A}} = \mathrm{Con}(A;s,1)$,
for every $\alg{A}\in\mathcal{V}$.
\end{enumerate}
\end{cor}

\begin{proof}
Let $\mathcal{V}$ be an orderable, subtractive variety with EDPC.
By Theorem~\ref{thm:1-orderable-subtr} we get that $\mathcal{V}$ is Fregean.
By Theorems~3.1 and~3.4 of~\cite{AU97a}  
we get that for all $\alg{A}\in\mathcal{V}$ and all
$a,b\in A$ we have $(1,b)\in \theta(1,a)$ if and only if $s(a,b) = 1$
for a subtractive term $s$. This means Theorem~\ref{thm:oprema}(1)
is satisfied, and therefore so is Theorem~\ref{thm:oprema}(3). 
By Lemma~\ref{lem:unique} we obtain the uniqueness of $s$, and thus (2) follows
proving that (1) implies (2). The converse implication is obvious. 
\end{proof}

\section{A representation}

We begin by showing that pure Hilbert algebras form skeletons of all finite Hilbert
algebras in the sense that will become clear as we proceed.  
For any $\alg{A}\in\var{Hi}$, we 
define the set $\irr{\alg{A}}$ of \emph{irreducible elements}, putting 
$\irr{\alg{A}} = \{a\in A: \theta(1,a)\in \J{\Con{\alg{A}}}\}$
where $\J{\Con{\alg{A}}}$ is the poset of join-irreducible congruences of
$\Con{\alg{A}}$ ordered by inclusion. If $\alg{A}$ is finite,
each $\alpha\in \J{\Con{\alg{A}}}$ is completely join-irreducible, and
there exists a unique $a\in A$ such that $\alpha = \theta(1,a)$.
Hence, $\alpha$ has a unique subcover $\alpha_-$.
We will use this fact in proofs, without further notice. 

In finite Hilbert algebras and Brouwerian semilattices
the irreducible elements have an 
intrinsic characterisation given below. It was proved for Hilbert algebras,
by more complicated arguments, in~\cite{PS17} and in~\cite{CC05}, and for
Brouwerian semilattices in~\cite{Koh81}.

\begin{theorem}\label{thm:irreducible-charact}
Let $\alg{A}\in\var{Hi}$ be finite. Then for any $a\in A\smallsetminus\{1\}$,
the following are equivalent:
\begin{enumerate}
\item $a\in \irr{\alg{A}}$,  
\item $\theta(1,a)\in \J{\Con{\alg{A}}}$,
\item $\forall m\in A: m\ra a\in\{1,a\}$. 
\end{enumerate}
If $\alg{A}\in\var{Br}$ is finite, then the conditions above are
further equivalent to
\begin{enumerate}
\setcounter{enumi}{3}  
\item $a$ is meet-irreducible. 
\end{enumerate}  
\end{theorem}

\begin{proof}
The equivalence of (1) and (2) is definitional. Now let
$\theta(1,a)\in\J{\Con{\alg{A}}}$ and take any $m\in A$. Then,
$\theta(1,a)\subseteq \theta(1,m\ra a)\vee\theta(1,m)$. By congruence
distributivity, $\theta(1,a) =
\bigl(\theta(1,m\ra a)\cap\theta(1,a))\vee(\theta(1,m)\cap\theta(1,a)\bigr)$.
Since $\theta(1,a)\in\J{\Con{\alg{A}}}$, we get
$\theta(1,a)\subseteq \theta(1,m\ra a)$ or
$\theta(1,a)\subseteq \theta(1,m)$. Hence $m\ra a\leq a$ or
$m\leq a$, and so $m\ra a = a$ or $m\ra a = 1$, proving that (2) implies (3).
For the converse, first we show that the set $F_a = \{b\geq a: b\ra a = a\}$ is
a filter. 
Take $c,d\in A$ with $c, c\ra d\in F_a$. Then
$c\ra a = a = (c\ra d)\ra a$, so we get
$d\ra a = d\ra (c\ra a) = c\ra (d\ra a) = (c\ra d)\ra(c\ra a) =
(c\ra d)\ra a = a$. As $a\leq c$ and $a\leq c\ra d$, we get 
$a\leq d$, so $d\in F_a$ as needed.

Now, since $a\notin F_a$ we have 
$F_a\subsetneq 1/\theta(1,a)$. Take $\alpha\in \Con{\alg{A}}$ such that
$\alpha\subsetneq \theta(1,a)$. Then $m>a$ for every $m\in 1/\alpha$,
so $m\ra a = a$ by (3), and hence,
$1/\alpha\subseteq F_a$. As $\alpha$ was arbitrary, we obtain
$\bigvee\{\alpha\in\Con{\alg{A}}: \alpha\subsetneq \theta(1,a)\}
\subseteq \theta(F_a) \subsetneq \theta(1,a)$ showing that
$\theta(1,a)\in \J{\Con{\alg{A}}}$; indeed, $\theta(F_a) = \theta(1,a)_-$.

The equivalence of (3) and (4) in Brouwerian semilattices follows from
the fact that $a< b$ and $a < b\ra a < 1$ hold if and only if
$b\wedge(b\ra a) = a$.
\end{proof}

Now, for any finite $\alg{A}\in\var{Hi}$ we define the algebra
$\ov{\Irr{\alg{A}}} = (\ov{\irr{\alg{A}}}; \ra, 1)$. Clearly,
$\ov{\Irr{\alg{A}}}$ is a pure Hilbert algebra. 

\begin{theorem}\label{thm:repr-fin-Hi}
Let $\alg{A}\in\var{Hi}$ be finite. Then $\ov{\Irr{\alg{A}}}\leq \alg{A}$ and
$\Con{\ov{\Irr{\alg{A}}}} \cong \Con{\alg{A}}$. 
\end{theorem}

\begin{proof}
By Theorem~\ref{thm:irreducible-charact} $\ov{\Irr{\alg{A}}}$ is a 
subalgebra of $\alg{A}$. Next, if $a\in\irr{\alg{A}}$, then by 
Theorem~\ref{thm:irreducible-charact} again, 
$\theta^{\ov{\Irr{\alg{A}}}}(1,a) \in \J{\Con{\ov{\Irr{\alg{A}}}}}$. By EDPC
(in fact CEP suffices) we get $\theta^{\ov{\Irr{\alg{A}}}}(1,a) =
\theta^{\alg{A}}(1,a)\cap (\ov{\Irr{\alg{A}}}\times\ov{\Irr{\alg{A}}})$  
showing that the posets $(\J{\Con{\ov{\Irr{\alg{A}}}}}; \subseteq)$ and
$(\J{\Con{\alg{A}}}; \subseteq)$ are isomorphic,
which implies $\Con{\ov{\Irr{\alg{A}}}} \cong
\Con{\alg{A}}$.
\end{proof}

Now we generalise the construction of pure Hilbert algebras from posets
to obtain a representation of all finite Hilbert algebras. It can be extended to
all Hilbert algebras, but at the cost of introducing a topology.

\begin{defn}\label{def:ant-order}
Let $P$ be any poset, and let $\mathrm{Ac}(P)$ be the set of all
antichains in $P$. Define:
\begin{itemize}
\item a binary relation $\preccurlyeq$ on $\mathrm{Ac}(P)$ putting
$C\preccurlyeq B$ if $B\subseteq \dw{C}$,
\item a binary operation $\ra$ on $\mathrm{Ac}(P)$ putting 
$C\ra B = \{b \in B: \up{b}\cap C = \emptyset\}$,
\item a binary operation $\wedge$ on $\mathrm{Ac}(P)$ putting 
  $B\wedge C = \max(B\cup C)$.
\end{itemize}  
\end{defn}

The relation $\preccurlyeq$ is clearly an order, in fact a semilattice order with
$\wedge$ being the meet operation. If $B =\{b\}$ and $C=\{c\}$, then
$C\preccurlyeq B$ if and only if $b\leq c$, so $\preccurlyeq$ extends the dual of
the native $\leq$ on $P$ to $\mathrm{Ac}(P)$.
Furthermore, $\ra$ is a Hilbert algebra operation, and for
$B =\{b\}$ and $C=\{c\}$ we have
$$
C\ra B =\begin{cases}
  \emptyset & \text{ if }b\leq c,\\ 
  B & \text{ otherwise.}
\end{cases}
$$

\begin{defn}\label{def:antichain-alg}
For any poset $P$ we define algebras
\begin{itemize}  
\item $\alg{Ac}(P) = (\mathrm{Ac}(P); \ra,\emptyset)$,
\item $\alg{Acm}(P) = (\mathrm{Ac}(P); \wedge,\ra,\emptyset)$,
\end{itemize}
\end{defn}

\begin{lemma}\label{lem:antichain-alg}
Let $P$ be any poset. Then,
\begin{enumerate}
\item The algebra $\alg{Ac}(P)$ is a Hilbert algebra.
\item The algebra $\alg{Acm}(P)$ is a Brouwerian semilattice. 
\end{enumerate}
Moreover, the map $\varphi(A) = P\smallsetminus \dw{A}$ for any
$A\in\mathrm{Ac}(P)$, is an embedding
of ${(\mathrm{Ac}(P), \preccurlyeq)}$ into $(\mathrm{Up}(P); \subseteq)$.
If $P$ is finite, then $\varphi$ is an isomorphism.
\end{lemma}

\begin{proof}
(1) and (2) follow immediately from Definition~\ref{def:ant-order}, the
remarks following it, and Definition~\ref{def:antichain-alg}. The `moreover' part
follows easily upon observing that for any $U\in\mathrm{Up}(P)$ the set
$\varphi^{-1}(U)$ is the antichain of all maximal elements of $P\smallsetminus U$. 
\end{proof}

Lemma~\ref{lem:antichain-alg} suggests that a generic construction of a Hilbert
algebra from a poset $P$ will use some selection of antichains of $P$, in
contrast to an analogous construction of a Brouwerian semilattice, which will
use all antichains. Indeed it is the case, as the theorem below shows.
In its proof, and later on, for any $\alg{A}\in\{\var{Hi},\var{Br}\}$
and any $\mu\in\cm{\alg{A}}$, we will write $\op_\mu$ for the 
unique coatom of 
$\alg{A}/\mu$. Clearly, $a/\mu = \op_\mu\iff \mu\vee\theta(a,1) = \mu^+$.
We represent $\alg{A}$ as an algebra of antichains of $\Cm{\alg{A}}$.

\begin{theorem}[Representation]\label{thm:mono}
Let $\alg{A}\in\{\var{Hi},\var{Br}\}$.  
Let $S\colon\alg{A}\to\mathrm{Ac}(\Cm{\alg{A}})$ be defined by
\begin{equation}\label{eq:S}\tag{\S}
S(a) = \{\mu\in\cm{\alg{A}}:\mu\vee\theta(1,a) = \mu^+\}.
\end{equation}
Then, $S$ is a monomorphism, and 
\begin{enumerate}
\item  if $\alg{A}\in\var{Hi}$, then 
  $\alg{A}\cong S(\alg{A}) \leq \alg{Ac}(\Cm{\alg{A}})$;
\item  if $\alg{A}\in\var{Br}$, then 
  $\alg{A}\cong S(\alg{A}) \leq \alg{Acm}(\Cm{\alg{A}})$;
\item  if $\alg{A}\in\var{Br}$ is finite, then 
  $\alg{A}\cong S(\alg{A}) \cong \alg{Acm}(\Cm{\alg{A}})$.  
\end{enumerate}
\end{theorem}  

\begin{proof}
It is clear that $S(a)$ is an antichain for any $a\in A$; in particular, $S(1) =
\emptyset$.  Take $a,b\in A$ with $a\neq b$. Then, $\theta(a,b)\neq
\mathrm{Eq}^{\alg{A}}$, and by a standard Zorn lemma argument we get that there
exists some $\mu\in \cm{\alg{A}}$ such that $(a,b)\in\mu^+\smallsetminus\mu$. 
Then by Lemma~\ref{lem:1-orderable-si} it follows that
$\{a/\mu,b/\mu\} = \{1/\mu,\op_\mu\}$ and therefore either
$\mu\in S(a)\smallsetminus S(b)$ or $\mu\in S(b)\smallsetminus S(a)$ showing that
$S$ is injective. 

It is easy to see that for any $a,b\in A$ and any $\mu\in\cm{\alg{A}}$ we have
$\mu\in S(a\ra b)$ if and only if $(a\ra b)/\mu = a/\mu\ra b/\mu = \op_\mu$
if and only if $a/\mu = 1/\mu$ and $b/\mu = \op_\mu$. We will use this fact in
the next part of the proof.

To show $S(a\ra b)\subseteq S(a)\ra S(b)$ take $\mu$ such that
$a/\mu\ra b/\mu = \op_\mu$. Then $a/\mu = 1/\mu$ and $b/\mu = \op_\mu$. Hence 
$\mu\in S(b)$ and $\up{\mu}\cap S(a) = \emptyset$, 
as $\nu\notin S(a)$ for any $\nu \supseteq \mu$ with
$(1,a)\in\nu$. So $\mu\in S(a)\ra S(b)$.

For the converse inclusion, take $\mu\in S(a)\ra S(b)$ so that
$\mu\in S(b)$ and $\up{\mu}\cap S(a) = \emptyset$. If
$(1,a)\notin \mu$, then for some $\nu\supseteq \mu$ we have
$(1,a)\in\nu^+\smallsetminus \nu$, but this implies
$\nu\in \up{\mu}\cap S(a)$ which is a contradiction. 
Therefore, $(1,a)\in \mu$, that is $1/\mu = a/\mu$, and consequently
$a/\mu \ra b/\mu = (1\ra b)/\mu = \op_\mu$ as required. 

Finally, assume $\alg{A}\in\var{Br}$, so $a\wedge b$ is defined for all $a,b\in
A$. Note that $\theta(1, a\wedge b) = \theta(1,a)\vee\theta(1,b)$.
To prove $S(a\wedge b) \subseteq S(a)\wedge S(b)$ take $\mu\in S(a\wedge b)$
so that $\mu\vee\theta(1,a\wedge b) = \mu\vee\theta(1,a)\vee\theta(1,b) =
\mu^+$. If $\mu\vee\theta(1,a)<\mu^+$ and $\mu\vee\theta(1,b)<\mu^+$, then
$\mu\vee\theta(1,a)\vee\theta(1,b) = \mu$, contradicting the assumption, so
without loss of generality $\mu\vee\theta(1,a) =\mu^+$ and therefore
$\mu\in S(a)$. If $\mu$ is not maximal in $S(a)\cup S(b)$, then there is a
$\nu\in S(b)$ such that $\mu<\nu$, since $S(a)$ is an antichain.
In that case,
$\mu\vee\theta(1,a)\vee\theta(1,b) = \mu^+\subseteq \nu$,
so we obtain $(1,b)\in \nu$, a contradiction. 

For the inclusion $S(a)\wedge S(b) \subseteq S(a\wedge b)$ take
$\mu \in S(a)\wedge S(b)$ so that $\mu$ is maximal in $S(a)\cup S(b)$.
Without loss of generality, assume $\mu \in S(a)$, that is,
$\mu\vee\theta(1,a) = \mu^+$.
We claim that $\mu\vee\theta(1,b) \subseteq \mu^+$.
Suppose $\mu\vee\theta(1,b) \supsetneq \mu^+$. Let $\nu\in\Cm{\alg{A}}$
be such that $\nu\geq\mu^+$ and $(1,b)\in\nu^+\smallsetminus\nu$. 
Then $\nu\vee\theta(1,b) = \nu^+$, so $\nu\in S(b)$,
and $\nu\supsetneq\mu$ contradicting the assumption that $\mu$ is maximal
in $S(a)\cup S(b)$. Hence, $\mu\vee\theta(1,b)\subseteq \mu^+$ as we claimed, and
therefore $\mu\vee\theta(1,a)\vee \theta(1,b) = \mu^+$, showing that
$\mu\in S(a\wedge b)$.

Then, (1) and (2) follow trivially, and (3) follows by the `moreover' part of
Lemma~\ref{lem:antichain-alg} and the standard representation of
Brouwerian semilattices as algebras of upsets.
\end{proof}  

\begin{theorem}\label{thm:repres} 
Let $\alg{A}\in\var{Hi}$. If $\alg{A}$ is finite, then
\begin{enumerate}
\item $S(\Irr{\alg{A}}) = \{\{\mu\}: \mu\in\cm{\alg{A}}\}$,\label{a}
\item $\Con{\mathbf{Ac}(\Cm{\alg{A}})} = \Con{\mathbf{Acm}(\Cm{\alg{A}})} 
  \cong\Con{\alg{A}}$.\label{b}
\end{enumerate}    
\end{theorem}

\begin{proof}
For~\eqref{a} let $a\in \irr{\alg{A}}$. Then there exists a $\mu\in\Cm{\alg{A}}$
such that $(1,a)\in \mu^+\smallsetminus\mu$, hence $\mu\in S(a)$. Seeking a
contradiction assume $|S(a)|\geq 2$ and take $\nu\in S(a)\smallsetminus\{\mu\}$. 
Since $S(a)$ is an antichain, we have $\mu\not\subseteq\nu$, and so
$\nu^+\subseteq \nu\vee\mu$. Hence, $\theta(1,a)\subseteq \nu\vee\mu$.
On the other hand, since $\nu,\mu\in S(a)$ we have
$\theta(1,a)\cap\nu\subseteq \theta(1,a)_-$ and 
$\theta(1,a)\cap\mu\subseteq \theta(1,a)_-$, so
$(\theta(1,a)\cap\nu)\vee(\theta(1,a)\cap\mu)\subseteq \theta(1,a)_-$.
Then, using the fact that $\var{Hi}$ is CD, we obtain
$\theta(1,a) = \theta(1,a)\cap(\nu\vee\mu) =
(\theta(1,a)\cap\nu)\vee(\theta(1,a)\cap\mu) \subseteq \theta(1,a)_-$ which is a
desired contradiction.

It follows that for each $a\in \irr{\alg{A}}$ we have exactly one
$\mu\in\Cm{\alg{A}}$ such that $a/\mu = \op_\mu$. Denoting this unique $\mu$
associated with $a$ by $\mu_a$, we have that $S(a) = \{\mu_a\}$.
Note that for any $\mu\in\cm{\alg{A}}$ there exists $\alpha\in
\J{\Con{\alg{A}}}$ such that $\alpha$ and $\mu$ are complements in the interval
$\mathrm{I}[\alpha_-, \mu^+]$. From this, since $\alpha = \theta(1,a)$ for some $a\in
\irr{\alg{A}}$, we obtain, using projectivity of intervals, that
$S(\Irr{\alg{A}}) = \{\{\mu\}: \mu\in\cm{\alg{A}}\}$ as required.

For~\eqref{b}, the first equality follows by Theorem~\ref{thm:oprema}, so
we only need to prove the second. To simplify notation put
$\alg{B} = \mathbf{Acm}(\Cm{\alg{A}})$.
By Theorem~\ref{thm:oprema} again, we obtain 
$\Con{\alg{B}} = \mathbf{Con}(B;\ra,1)$
and so $\J{\Con{\alg{B}}} = \J{\mathbf{Con}(B;\ra,1)}$,
from which it follows that
$\irr{\alg{B}} = \mathrm{I}(B;\ra,1) = 
\irr{\mathbf{Ac}(\Cm{\alg{A}})}$
as posets. From the definition of $\ra$ it is immediate that
$C\in \irr{\mathbf{Ac}(\Cm{\alg{A}})}$ if and only if $|C| = 1$.
Hence, by~\eqref{a},  
$\ov{\irr{\alg{A}}} \cong \ov{\irr{\alg{B}}}$
and then using Theorem~\ref{thm:repr-fin-Hi} we get the desired 
$\Con{\alg{A}}\cong\Con{\alg{B}}$.
\end{proof}  

The next result is crucial for our free algebra construction. It shows that
values of the map $S$ are closed under subsets.

\begin{theorem}\label{thm:transit} 
Let $\alg{A}\in\mathcal{V}$, where $\mathcal{V}\in \{\var{Hi},\var{Br}\}$.
If $\alg{A}$ is finite, then
for any $a\in A$ and any $L\subseteq \cm{\alg{A}}$ we have
$L\subseteq S(a)\implies L\in S(\alg{A})$.
\end{theorem}

\begin{proof}
Fix some $a\in A$ and $L\subseteq S(a)$. By~Theorem~\ref{thm:repres}\eqref{a} we know
that $\cm{\alg{A}} =  \{\mu_a: a\in \irr{\alg{A}}\}$, so we can assume
$S(a) = \{\mu_{a_1},\dots,\mu_{a_n}\}$ for some antichain
$\{a_1,\dots,a_n\}\subseteq \irr{\alg{A}}$, and that $L =
\{\mu_{a_1},\dots,\mu_{a_k}\}$.  Then
\begin{align*}
  S\bigl(a_{k+1}\ra(a_{k+2}\ra( \dots (a_n\ra a)\dots\bigr) &= \\
\{\mu_{a_{k+1}}\}\ra\bigl(\{\mu_{a_{k+2}}\}\ra(\dots (\{\mu_{a_n}\}\ra
  S(a))\dots\bigr) &= \\
\{\mu_{a_{k+1}}\}\ra\bigl(\{\mu_{a_{k+2}}\}\ra(\dots (\{\mu_{a_n}\}\ra  
  \{\mu_{a_1},\dots,\mu_{a_k}\})\dots\bigr) &= \\
  \{\mu_{a_1},\dots,\mu_{a_n}\}\smallsetminus\{\mu_{a_{k+1}},\dots,\mu_{a_n}\} &= \\
\{\mu_{a_1},\dots,\mu_{a_k}\} &= L   \qedhere
\end{align*}  
\end{proof}

\section{Free algebra construction}\label{sec:free-Hi}

\subsection{Varieties of Hilbert algebras}
Let $\mathcal{V}$ be a variety of Hilbert algebras.
We will construct the free $\mathcal{V}$-algebra
$\FV(n)$ from the poset of its completely meet-irreducible congruences,
itself constructed inductively by layers.
Let $X = \{x_1,\dots, x_n\}$ be the set of free generators of
$\FV(n)$. We begin by decomposing $\Cm{\FV(n)}$ into layers 
$$ 
P_k(n) = \{\mu\in \cm{\FV(n)}: h(\mu) = k\}
$$
where for any finite algebra $\alg{A}$ the map
$h\colon \cm{\alg{A}}\to \mathbb{N}$ is the map giving the height of
$\mu$, that is, the length of the longest chain in $\up{\mu}\subseteq
\cm{\alg{A}}$. The same map applies to algebras, $h(\alg{A})$ being the 
length of the longest chain in $\Cm{\alg{A}}$. Trivially, then
$$
\cm{\FV(n)} = \bigcup_{k=1}^\ell P_k(n)
$$
for some $\ell$, in general depending on $\mathcal{V}$.
We will see shortly that $\ell$ cannot exceed $n$. 

\begin{lemma}\label{lem:P1}
For any $J\subsetneq X$ let
$f_J\colon X\to\{0,1\}$ be the map given by
$$
f_J(x) =\begin{cases}
  1 & \text{ if } x\in J,\\
  0 & \text{ if } x\notin J,
\end{cases}
$$
where $\{0,1\}$ is the universe of $\alg{2}$.
Let $\ov{f}_J\colon \FV(n)\twoheadrightarrow \mathbf{2}$
be the homomorphism uniquely extending $f_J$, and let
$\mu_J = \ker \ov{f}_J$. Then
$P_1(n) = \{\mu_J: J\subsetneq X\}$.
\end{lemma}  
  
\begin{proof}
Clearly, $\mu\in P_1(n)$ if and only if
$\FV(n)/\mu \cong \alg{2}$. It is also clear that
for $\mu,\nu\in P_1(n)$ we have $\mu\neq\nu$ if and only if
$\mu$ and $\nu$ are kernels of homomorphisms
extending, respectively, maps $f_\mu\colon X\to \{0,1\}$ and
$f_\nu\colon X\to \{0,1\}$. These maps must be surjective as otherwise
$\FV(n)/\mu$ or $\FV(n)/\nu$ would be the trivial algebra. Hence
$J$ must be a proper subset of $X$. 
\end{proof}

Note that the above construction of $P_1(n)$ does not depend on $\mathcal{V}$
at all: it is the same for all subvarieties of $\var{Hi}$. This is due to the
fact that the two-element Tarski algebra belongs to every nontrivial
$\mathcal{V}$. Specific features of $\mathcal{V}$ 
start playing a role from layer $P_2(n)$,
since some subdirectly irreducible algebras of height
$\geq 2$ may not belong to $\mathcal{V}$. To keep our reasoning as neat as
possible we will first consider the case $\mathcal{V} = \var{Hi}$, and then
generalise.

\begin{lemma}\label{lem:layers}
Let $\mathcal{V} = \var{Hi}$ and $k\geq 2$.
For any $G\in \mathrm{Up}\left(\bigcup_{r=1}^{k-1}P_r(n)\right)$
such that $G\cap P_{k-1}(n)\neq \emptyset$, and for any $L$ such that
$L\subsetneq X\cap 1/\bigcap G$,
let $f^G_L\colon X \to  \bigl(\FV(n)/\bigcap G\bigr)^\oplus$ be the map
given by 
$$
f^G_L(x) =\begin{cases}
  1 & \text{ if } x\in L,\\
  \op & \text{ if } x\in \left(1/\bigcap G\right)\smallsetminus L,\\
  x/\bigcap G & \text{ if } x\notin 1/\bigcap G.
\end{cases} 
$$
Let $\ov{f}^G_L$ be the homomorphism uniquely extending $f^G_L$ and
let $\mu^G_L = \ker\ov{f}^G_L$. Then
$$
P_k(n) =
\left\{\mu^G_L: G\in \mathrm{Up}\left(\bigcup_{r=1}^{k-1}P_r(n)\right),\ 
G\cap P_{k-1}(n)\neq \emptyset, \
L\subsetneq X\cap 1/\bigcap G  
\right\}.
$$
Furthermore, for $r>n$ we have $P_r(n) = \emptyset$; hence
$\bigcup_{k=1}^\infty P_k(n) = \bigcup_{k=1}^{n} P_k(n)$.
\end{lemma}  

\begin{proof}
We proceed by induction on $k$. For $k=1$, the claim follows by
Lemma~\ref{lem:P1}.
Now, assume inductively that for any $r\in \{1,\dots, k-1\}$
the set $P_r(n)$ satisfies the claim. 
Take any $G\in \mathrm{Up}\bigl(\bigcup_{r=1}^{k-1}
P_r(n)\bigr)$ such that $G\cap P_{k-1}(n)\neq \emptyset$, take any
$L\subsetneq \bigcap\{X\cap 1/\varphi: \varphi \in G\}
= X\cap 1/\bigcap G$.
To simplify notation put $\gamma = \bigcap G$.
Then, reasoning as in the base case, we obtain that
\begin{itemize}
\item[(i)] $\gamma\in \con{\FV(n)}$ is of height $k-1$,
\item[(ii)] $X\cap 1/\gamma = \bigcap\{X\cap 1/\varphi: \varphi \in G\}$, and
\item[(iii)] $\FV(n)/\gamma$ is generated by
  $\{x_i/\gamma: x_i\notin 1/\gamma\}$.
\end{itemize}  
Let $f^G_L\colon X \to  \left(\FV(n)/\gamma\right)^\oplus$
be given by 
$$
f^G_L(x_i) =\begin{cases}
  1 & \text{ if } x_i\in L\\
  \op & \text{ if } x_i\in 1/\gamma\smallsetminus L\\
  x_i/\gamma & \text{ if } x_i\notin 1/\gamma
\end{cases}
$$
Then, the homomorphism
$\ov{f}^G_L\colon \FV(n) \to \left(\FV(n)/\gamma\right)^\oplus$
uniquely extending $f^G_L$ is an epimorphism.  Let $\mu^G_L = \ker\ov{f}^G_L$.
The following hold by construction:
\begin{itemize}
\item[(i)] $\mu^G_L\in P_k(n)$,
\item[(ii)] $(\mu^G_L)^+ = \gamma$, and
\item[(iii)] $L = X\cap 1/\mu^G_L$.
\end{itemize}
This shows the right-to-left inclusion.

For the converse inclusion, take $\mu\in P_k(n)$, and let  
$G = \{\varphi\in\cm{\FV(n)}: \mu<\varphi\}$. Then
$G\in\mathrm{Up}\bigl(\bigcup_{r=1}^{k-1}P_r(n)\bigr)$ and 
$G\cap P_{k-1}(n)\neq\emptyset$. We know that 
$x_i/\mu = \op_\mu$ for some $x_i\in X$, and that
$1/\mu^+ = 1/\mu\cup\{\op_\mu\} = 1/\mu\cup x_i/\mu$. 
Let $L = X\cap 1/\mu$. Then, 
$L\subsetneq X\cap 1/\mu^+ =  X\cap 1/\bigcap G =
\bigcap\{X\cap 1/\varphi: \mu<\varphi\}$. 
Therefore, $\mu = \mu^G_L$, finishing the proof of the main part.

For the `furthermore' part, consider a chain
$\mu_1\supsetneq\dots\supsetneq\mu_r$ of congruences
in $\Cm{\FV(n)}$. It induces a chain
$\FV(n)/\mu_r\onto \dots \onto \FV(n)/\mu_1\cong \alg{2}$
of subdirectly irreducible quotients of $\FV(n)$.  Then
for each $j \in \{1,\dots,r\}$ we have $x_i/\mu_j = \op_{\mu_j}$ for some $i\in
\{1,\dots,n\}$, and moreover $x_i/\mu_\ell = 1/\mu_\ell$ for every
$\ell>j$; in particular, distinct $j$ require distinct $i$. By pigeonhole principle $r \leq n$.
\end{proof}  

Now we generalise Lemma~\ref{lem:layers} to an arbitrary variety
$\mathcal{V}\subseteq\var{Hi}$. A crucial ingredient in the lemma is
the map $f^G_L\colon X \to  \bigl(\FV(n)/\bigcap G\bigr)^\oplus$, but this map
is well defined only if $\bigl(\FV(n)/\bigcap G\bigr)^\oplus\in\mathcal{V}$.
We need to take it into account.

\begin{lemma}\label{lem:layers-gen}
Let $\mathcal{V}\subseteq\var{Hi}$ and $k\geq 2$. For any
$G\in \mathrm{Up}\left(\bigcup_{r=1}^{k-1}P_r(n)\right)$ 
such that $G\cap P_{k-1}(n)\neq \emptyset$ and
$\bigl(\FV(n)/\bigcap G\bigr)^\oplus\in\mathcal{V}$,
and for any $L$ such that $L \subsetneq X\cap 1/\bigcap G$,
let $f^G_L\colon X \to  \bigl(\FV(n)/\bigcap G\bigr)^\oplus$ be the map
given by 
$$
f^G_L(x) =\begin{cases}
  1 & \text{ if } x\in L,\\
  \op & \text{ if } x\in \left(1/\bigcap G\right)\smallsetminus L,\\
  x/\bigcap G & \text{ if } x\notin 1/\bigcap G.
\end{cases} 
$$
Let $\ov{f}^G_L$ be the homomorphism uniquely extending $f^G_L$ and
let $\mu^G_L = \ker\ov{f}^G_L$. Then
\begin{multline*}
P_k(n) =
\biggl\{\mu^G_L: G\in \mathrm{Up}\left(\bigcup_{r=1}^{k-1}P_r(n)\right),\ 
G\cap P_{k-1}(n)\neq \emptyset, \\
\bigl(\FV(n)/\textstyle\bigcap G\bigr)^\oplus\in\mathcal{V}, \
L\subsetneq X\cap 1/\bigcap G  
\biggr\}.
\end{multline*}
Furthermore, let $m = \min\{n,\ell\}$, where $\ell$ is the length of the longest chain
in $\Cm{\FV(n)}$. Then, for $r>m$ we have $P_r(n) = \emptyset$; hence
$\bigcup_{k=1}^\infty P_k(n) = \bigcup_{k=1}^{m} P_k(n)$.
\end{lemma}  

\begin{proof}
The proof of the main part is identical to the proof of Lemma~\ref{lem:layers}.
For the `furthermore' part, if $\ell \geq n$, then the proof of 
Lemma~\ref{lem:layers} applies. Assume $m = \ell<n$ and consider
$G\in \mathrm{Up}\left(\bigcup_{r=1}^{m}P_r(n)\right)$ 
such that $G\cap P_{m}(n)\neq \emptyset$. Then
$h\bigl(\FV(n)/\bigcap G\bigr) = \ell$, and
therefore $\bigl(\FV(n)/\bigcap G\bigr)^\oplus\notin\mathcal{V}$.
Hence the construction of layers stops at $P_m(n)$, that is,
$P_{m+1}(n) = \emptyset$. 
\end{proof}

Now, recall the map $S\colon\alg{A}\to\alg{Ac}(\Cm{\alg{A}})$ defined
in Theorem~\ref{thm:mono} by~\eqref{eq:S}. Applying it to
$\FV(n)$, we get
$S(x) = \{\mu\in\cm{\FV(n)}:\mu\vee\theta(1,x) = \mu^+\}$,
for any $x\in X$. Equivalently,
$$
S(x) = \left\{\mu\in\cm{\FV(n)}: x\notin 1/\mu \mathbin{\&}
x\in \bigcap\{1/\nu: \nu\supsetneq\mu\}\right\}
$$
so $S$ selects certain antichains from $\Cm{\FV(n)}$.

\begin{theorem}[Free Hilbert algebras]\label{thm:free-Hilbert-alg}
Let $\mathcal{V}\subseteq\var{Hi}$. For any $n > 0$, we have
$$
\FV(n)\cong \left(\bigcup_{i=1}^n\pw(S(x_i));\ra,\emptyset\right)
\leq \alg{Ac}(\Cm{\FV}(n)).
$$
\end{theorem}  

\begin{proof}
We show that $S(\FV(n)) =\bigcup_{i=1}^n\pw(S(x_i))$. The 
right-to-left inclusion follows from Theorem~\ref{thm:transit}.
For the reverse inclusion we proceed by induction on complexity of a term
$t\in\FV(n)$. The base holds by definition, so
let $t = t_1\ra t_2$. Then $S(t_1\ra t_2) = S(t_1)\ra S(t_2)\subseteq S(t_2)$ by
definition of $\ra$ in $\alg{Ac}(\Cm{\FV(n)})$, and so by the inductive
hypothesis $S(t_2)\subseteq S(x)$ for some $x\in X$, as required.
Now to finish the proof we invoke Theorem~\ref{thm:mono}.  
\end{proof}

\subsection{Varieties of Brouwerian semilattices}

Now consider a variety $\mathcal{W}$ of Brouwerian semilattices.
First, we show that the posets
$\Cm{\FW}(n)$ and $\Cm{\alg{F}}_{f(\mathcal{W})}(n)$ are isomorphic,
where $f:\Lambda(\var{Br})\to\Lambda(\var{Hi})$ is the embedding of
Theorem~\ref{thm:subvariety-isom}. Indeed, we will show more.

\begin{theorem}\label{thm:same-con}
Let $\mathcal{W}\in\Lambda(\var{Br})$. Then,
$\langle \alg{F}_{f(\mathcal{W})}(n)\rangle \cong
\FW(n)$. Hence, $\Con{\FW}(n)\cong \Con{\alg{F}}_{f(\mathcal{W})}(n)$, and obviously  
$\Cm{\FW}(n)\cong \Cm{\alg{F}}_{f(\mathcal{W})}(n)$.  
\end{theorem}  

\begin{proof}
Let $\mathcal{W}\in\Lambda(\var{Br})$. Let $x_1,\dots,x_n$ be the free
generators of $\alg{F}_{f(\mathcal{W})}(n)$ and $y_1,\dots,y_n$ be the free
generators of $\FW(n)$. Since $\FW(n)$ belongs to
$\mathcal{W}$, we get that $\FW^\ra(n)\in f(\mathcal{W})$.
Consider the map
$\varphi\colon \{x_1,\dots,x_n\}\to \FW^\ra(n)$, given by
$\varphi(x_i) = y_i$. We claim that the unique homomorphism
$\overline{\varphi}\colon \alg{F}_{f(\mathcal{W})}(n)\to \FW^\ra(n)$ extending
$\varphi$ is an injection. To see it, take $t\in \mathrm{Term}_{\{\ra,1\}}(n)$
such that $\overline{\varphi}(t^{\alg{F}_{f(\mathcal{W})}(n)}(x_1,\dots,x_n)) = 1$.
Then, $t^{\FW(n)}(y_1,\dots,y_n) = 1$, and so
$\mathcal{W}\models t =1$. Hence $f(\mathcal{W})\models t =1$, showing
that $\overline{\varphi}$ is injective, that is, an embedding.
Hence, there exist a unique
$\overline{\overline{\varphi}}\colon
\langle \alg{F}_{f(\mathcal{W})}(n) \rangle \hookrightarrow \FW(n)$
such that $\overline{\overline{\varphi}}|_{\alg{F}_{f(\mathcal{W})}(n)} =
\overline{\varphi}$ and
$\langle \alg{F}_{f(\mathcal{W})}(n) \rangle \cong
\overline{\overline{\varphi}}(\langle \alg{F}_{f(\mathcal{W})}(n) \rangle)$.
Since $y_1,\dots, y_n\in \overline{\overline{\varphi}}(\langle
\alg{F}_{f(\mathcal{W})}(n) \rangle)$ we obtain
$\langle \alg{F}_{f(\mathcal{W})}(n)\rangle \cong
\FW(n)$. 
\end{proof}

Now, for any $\mathcal{V}\in\var{Hi}$ such that
$\mathcal{V} = f(\mathcal{W})$ for some $\mathcal{W}\in\var{Br}$,
we will write $\widehat{\mathcal{V}}$ for $f^{-1}(\mathcal{V})$. 
By Theorem~\ref{thm:same-con}, the posets $\Cm{\FV(n)}$
and $\Cm{\FVhat(n)}$ are isomorphic. To obtain an analogue
of Theorem~\ref{thm:free-Hilbert-alg} we must take the free extension
of $\FV(n)$.

\begin{theorem}[Free Brouwerian semilattices]\label{thm:free-Br-semi}
For any $n > 0$, we have
$$
\FVhat(n)\cong
\left\langle\bigcup_{i=1}^n\pw(S(x_i));\ra,\wedge,\emptyset\right\rangle 
\cong \alg{Acm}(\Cm{\FV(n)}),
$$
where the angle brackets denote the free extension of\/
$\FV(n)$.
\end{theorem}

\section{Adding zero}\label{sec:zero}
Given an arbitrary
variety $\mathcal{V}\subseteq\var{Hi}$, or $\subseteq\var{Br}$,
we can expand the signature by a constant
$0$, and form $\mathcal{V}_0$ by adding the identity $0\ra x = 1$ to the base of
$\mathcal{V}$. We will denote $\var{Hi}_0$ by $\var{Ho}$ and
$\var{Br}_0$ by $\var{Bo}$ to avoid proliferation of subscripts later.  
We will refer to the algebras from $\var{Ho}$ and $\var{Bo}$ 
as \emph{Hilbert algebras with zero} and
\emph{Brouwerian semilattices with zero}, respectively. 

Let us fix an arbitrary variety
$\mathcal{V}_0$ of Hilbert algebras with zero. We will adapt our
construction of free algebras to suit $\mathcal{V}_0$.
The first modification consists in replacing Lemma~\ref{lem:P1} by the
following. 

\begin{lemma}\label{lem:P1-with-0}
For any $J\subseteq X$ let
$f_J\colon X\to\{0,1\}$ be the map given by
$$
f_J(x) =\begin{cases}
  1 & \text{ if } x\in J,\\
  0 & \text{ if } x\notin J,
\end{cases}
$$
where $\{0,1\}$ is the universe of $\alg{2}$.
Let $\ov{f}_J\colon \FVo(n)\twoheadrightarrow \mathbf{2}$
be the homomorphism uniquely extending $f_J$, and let
$\mu_J = \ker \ov{f}_J$. Then
$P_1(n) = \{\mu_J: J\subseteq X\}$.
\end{lemma}

\begin{proof}
Clearly, $\mu\in P_1(n)$ if and only if
$\FVo(n)/\mu \cong \alg{2}$. It is also clear that
for $\mu,\nu\in P_1(n)$ we have $\mu\neq\nu$ if and only if
$\mu$ and $\nu$ are kernels of homomorphisms
extending, respectively, maps $f_\mu\colon X\to \{0,1\}$ and
$f_\nu\colon X\to \{0,1\}$. Since $0$ is in the signature
$\FVo(n)/\mu$ is nontrivial even if  
$f_\mu$ is not surjective, and similarly for $\FVo(n)/\nu$ and $f_\nu$. Hence
$J$ can be any subset of $X$. 
\end{proof}

As a consequence of Lemma~\ref{lem:P1-with-0}, the first layer
$P_1(n)$ of $\Cm{\FVo(n)}$ will have one element more than
$P_1(n)$ of $\Cm{\FV(n)}$, namely the element corresponding to
$X$, which in turn corresponds to the homomorphism from $\FVo(n)$
onto $\alg{2}$ sending all generators to $1$. Such a homomorphism from
$\FV(n)$ clearly does not exist. Hence $|P_1(n)| = 2^n-1$ 
in $\Cm\FV(n)$, but $|P_1(n)| = 2^n$ in $\FVo(n)$. 
Then Lemma~\ref{lem:layers} applies without
changes to $\Cm{\FVo(n)}$, 
except that now we will have one more layer, so the number of layers will be
$n+1$ instead of $n$. Intuitively, subdirectly irreducible algebras can
grow one level taller, by adding $0$.

The second modification needed to obtain an analogue of
Theorem~\ref{thm:free-Hilbert-alg} is to  
add $S(0)$, since $0$ is not generated by any term other than itself. That is,
according to the definition of $S$, we get
$S(0) = \{\mu\in\cm{\FVo(n)}: \mu^+ = 1_{\FVo(n)}\} =  P_1(n)$,
  unlike in $\Cm{\FV(n)}$ where we have $S(t)
\neq P_1(n)$ for all $t\in \mathrm{Term}_{\ra}$. The next theorem is
proved exactly as Theorem~\ref{thm:free-Hilbert-alg}.

\begin{theorem}[Free Hilbert algebras with zero]\label{thm:free-Hi0-alg}
Let $\mathcal{V}_0\subseteq \var{Ho}$ be nontrivial.  
The algebra $\FVo(0)$ is term-equivalent to the two-element Boolean algebra. 
For any $n > 0$, we have
$$
\FVo(n)\cong
\left(\bigcup_{i=1}^{n}\pw(S(x_i))\cup \pw(S(0));\ra,\emptyset,P_1(n)\right)
\leq \alg{Ac}(\Cm{\FVo(n)}).
$$
\end{theorem}

Now we turn to Brouwerian semilattices with zero. It is evident from the proofs
that the analogues of  Lemmas~\ref{lem:min-ext}, and~\ref{lem:f-charact}, as well as 
of Theorem~\ref{thm:subvariety-isom} hold for
$\Lambda(\var{Bo})$ and $\Lambda(\var{Ho})$.
Adding $0$ to the signature does not change Theorem~\ref{thm:same-con} either.
We keep using the notation $\widehat{\mathcal{V}_0}$ analogously.

\begin{theorem}[Free Brouwerian semilattices with zero]\label{thm:free-Br-semi-0}
For any $n > 0$, we have
$$
\FVohat(n) \cong
\left\langle\bigcup_{i=1}^{n}\pw(S(x_i))\cup
  \pw(S(0));\ra,\emptyset,P_1(n)\right\rangle \cong \alg{Acm}(\Cm{\FVo}(n)),
$$
where the angle brackets denote the free extension of\/
$\FVo(n)$.
\end{theorem}  

Next we observe that free algebras in varieties with zero are isomorphic to
certain special quotients of free algebras in varieties without zero, but with
one more generator. The result is the same for Hilbert algebras and
Brouwerian semilattices.  

\begin{theorem}\label{thm:quotient}
Let $\mathcal{V}\subseteq\var{Hi}$ or $\mathcal{V}\subseteq\var{Br}$.
For any $n\geq 1$  the following holds
$$  
\FVo(n)\cong \FV(n+1)/\vartheta
$$
where $\vartheta = \bigvee_{i=1}^n\theta(x_{n+1}\ra x_i, 1)$.
\end{theorem}  

\begin{proof}
Let $X = \{x_i: i\leq n+1\}$. Consider the map $f\colon X\to \FVo(n)$ given
by $f(x_i) = x_i$ for $1\leq i\leq n$ and $f(x_{n+1}) = 0$. Then
$f$ extends to an onto homomorphism $\ov{f}: \FV(n+1)\to \FVo(n)$, such
that $\ov{f}(x_{n+1}\ra x_i) =1$ for any $i\leq n$. Hence,
$\vartheta\subseteq \ker{\ov{f}}$. 
To show the converse, let us inspect the algebra $\FV(n+1)/\vartheta$.
It is easy to show by induction on complexity of
$t\in\mathrm{Term}_{\ra}(n+1)$ or $t\in\mathrm{Term}_{\{\wedge,\ra\}}(n+1)$
that $x_{n+1}/\vartheta\ra t/\vartheta = 1/\vartheta$. It follows that
$x_{n+1}/\vartheta$ is the smallest element in $\FV(n+1)/\vartheta$.
To lighten the notation, put $\alg{A} = \FV(n+1)/\vartheta$,
$a_i = x_i/\vartheta$ for $i\leq n$ and 
$0^\alg{A} = x_{n+1}/\vartheta$. Then $\alg{A}$, in the signature expanded by
$0$, belongs to $\mathcal{V}_0$ and it is generated by $a_1,\dots,a_n$, so
$\alg{A}$ is a homomorphic image of $\FVo(n)$. Hence,
$|A|\leq |\mathit{F}_{\mathcal{V}_0}(n)|$. On the other hand, 
$|A|\geq |\mathit{F}_{\mathcal{V}}(n+1)/\ker{\ov{f}}| =
|\mathit{F}_{\mathcal{V}_0}(n)|$, and since 
$\alg{A} = \FV(n+1)/\vartheta$ we have
$|\mathit{F}_{\mathcal{V}}(n+1)/\vartheta| =
|\mathit{F}_{\mathcal{V}}(n+1)/\ker{\ov{f}}|$ and therefore
$\vartheta = \ker{\ov{f}}$.     
\end{proof}  

Theorem~\ref{thm:quotient} implies, via the constructions of
$\Cm{\FVo(n)}$ and $\Cm{\FV(n+1)}$ that
for any $n$ we have the following interlacing of posets:
$$
\Cm{\FV(n)}\mathbin{\up{\subseteq}} \Cm{\FVo(n)}
\mathbin{\up{\subseteq}} \Cm{\FV(n+1)}
$$
where $\mathbin{\up{\subseteq}}$ stands for `is contained as an upset in'. 
See Figures~\ref{fig:P3(3)} and~\ref{fig:Po3(2)} for an instance of this
with $n = 2$.

\section{Examples and applications}\label{sec:ex-appl}

We start with examples in pictures. We ask the reader to
compare the first two examples to~\cite{Koh81}; especially our Figure~\ref{fig:P3(3)}
to Figure~3 there. 

\begin{example}\label{ex:free-Hi2}
The two-generated free Hilbert algebra $\alg{F}_{\var{Hi}}(2)$ and the two-generated free
Brouwerian semilattice $\alg{F}_{\var{Br}}(2)$ are shown in Figure~\ref{fig:P2(2)}. 
\begin{figure}[h]
\begin{center}
\begin{tikzpicture}
\filldraw[black!15!white] (-9,2.7) rectangle (-4,3.4);
\filldraw[black!15!white] (-9,1.7) rectangle (-4,2.4);

\node[bdot] (u0) at (-8,2) {};
\node[sdot] (u1) at (-8,3) {};
\node[sdot] (u2) at (-7,2) {};
\node[bdot] (u3) at (-7,3) {};
\node[bdot] (u4) at (-6,3) {};
\node[sdot] at (u4) {};
\node[left] at (u1) {$\{x\}$};
\node[left] at (u3) {$\{y\}$};
\node (u5) at (-5,3) {$P_1(2)$};
\node (u6) at (-5,2) {$P_2(2)$};

\path
(u0) edge (u1) (u2) edge (u3);

\node[dot] (v0) at (0,0) {};
\node[bdot] (v1) at (1,1) {};
\node[right] at (v1) {$y$};
\node[bdot] (v2) at (2,2) {};
\node[bdot] (v3) at (-1,1) {};
\node[left] at (v3) {$x$};
\node[dot] (v4) at (0,2) {};
\node[bdot] (v5) at (1,3) {};
\node[bdot] (v6) at (-2,2) {};
\node[bdot] (v7) at (-1,3) {};
\node[bdot] (v8) at (0,4) {};

\node[dot] (v9) at (0,1) {};
\node[bdot] (v10) at (1,2) {};
\node[bdot] (v11) at (2,3) {};
\node[bdot] (v12) at (-1,2) {};
\node[dot] (v13) at (0,3) {};
\node[bdot] (v14) at (1,4) {};
\node[bdot] (v15) at (-2,3) {};
\node[bdot] (v16) at (-1,4) {};
\node[bdot] (v17) at (0,5) {};

\path
(v0) edge (v1) (v1) edge (v2)
(v3) edge (v4) (v4) edge (v5)
(v6) edge (v7) (v7) edge (v8)
(v0) edge (v3) (v3) edge (v6)
(v1) edge (v4) (v4) edge (v7)
(v2) edge (v5) (v5) edge (v8)

(v9) edge (v10) (v10) edge (v11)
(v12) edge (v13) (v13) edge (v14)
(v15) edge (v16) (v16) edge (v17)
(v9) edge (v12) (v12) edge (v15)
(v10) edge (v13) (v13) edge (v16)
(v11) edge (v14) (v14) edge (v17)

(v0) edge (v9)
(v1) edge (v10)
(v2) edge (v11)
(v3) edge (v12)
(v4) edge (v13)
(v5) edge (v14)
(v6) edge (v15)
(v7) edge (v16)
(v8) edge (v17)
;
\end{tikzpicture}  
\end{center}
\caption{Left: Hasse diagram of $P_1(2)\cup P_2(2)$.
Solid dots indicate $S(x)$, square dots $S(y)$; points
without labels are labelled by $\emptyset$.
Right: Hasse diagrams of $\alg{F}_{\var{Br}}(2)$ (all dots), and of $\alg{F}_{\var{Hi}}(2)$ (solid dots).}
\label{fig:P2(2)}   
\end{figure}
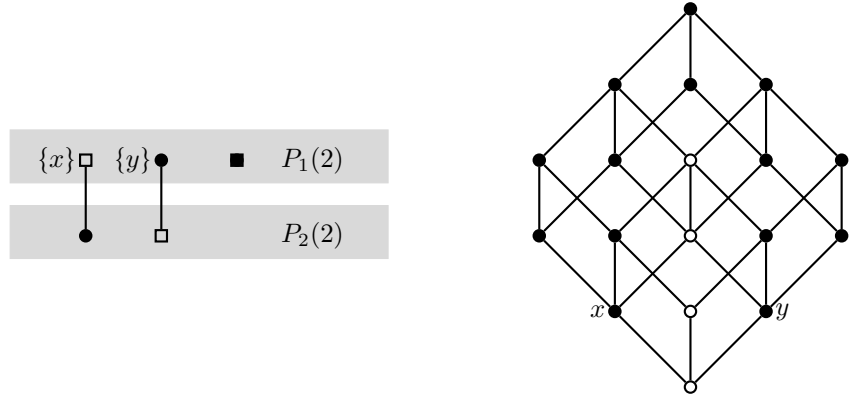
We have $|\Cm{\alg{F}_{\var{Hi}}(2)}| = |\Cm{\alg{F}_{\var{Br}}(2)}| = 5$, but
$|\alg{F}_{\var{Hi}}(2)| = 14$  and $|\alg{F}_{\var{Br}}(2)| = 18$, both numbers well known.
\end{example}

\begin{example}
The poset $\Cm{\alg{F}_{\var{Hi}}(3)}$ of the three-generated free Hilbert algebra, 
identical to the poset $\Cm{\alg{F}_{\var{Br}}(3)}$ of the three-generated free Brouwerian
semilattice is shown in Figure~\ref{fig:P3(3)} using a concentric Hasse diagram,
where the upward direction is towards the centre.  
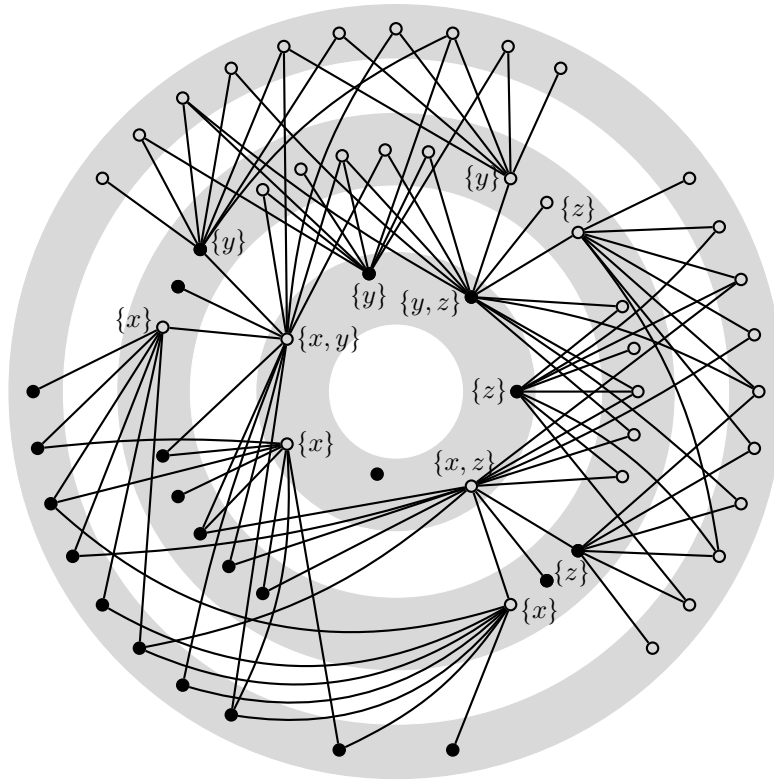
\begin{figure}
\begin{center}
\begin{tikzpicture}[scale=0.8]
\filldraw[black!15!white] (0,0) circle [radius=6.4];
\filldraw[white] (0,0) circle [radius=5.5]; 
  
\filldraw[black!15!white] (0,0) circle [radius=4.6];
\filldraw[white] (0,0) circle [radius=3.4]; 

\filldraw[black!15!white] (0,0) circle [radius=2.3];
\filldraw[white] (0,0) circle [radius=1.1]; 

\foreach \x in {1,2}
\node[dot] (v\x) at (\x*360/7:2) {};
\foreach \x in {3,4,6}
\node[dot] (v\x) at (\x*360/7:2) {};
\node[bdot] (v5) at (5*360/7:1.4) {};
\node[dot] (v7) at (7*360/7:2) {};
\node[right] at (v4) {$\{x\}$};
\node[right] at (v3) {$\{x,y\}$};
\node[bdot] (v2) at (2*360/7:2) {};
\node[below] at (v2) {$\{y\}$};
\node[bdot] (v1) at (1*360/7:2) {};
\node[left] at (v1.south) {$\{y,z\}$};
\node[bdot] (v7) at (7*360/7:2) {};
\node[left] at (v7) {$\{z\}$};
\node[above] at (v6.west) {$\{x,z\}$};

\foreach \x in {1,2,4,5,6,8,9,10,11,12,14,15,16,29,30,31,33,34,35}
\node[dot] (u\x) at (\x*360/35:4) {};
\foreach \x in {19,20,21,22,23}
\node[bdot] (u\x) at (\x*360/35:4) {};

\node[above] at (u4) {$\{z\}$};
\node[left] at (u6) {$\{y\}$};

\node[right] at (u14.north) {$\{y\}$};
\node[bdot] (u14) at (14*360/35:4) {};
\node[bdot] (u15) at (15*360/35:4) {};
\node[left] at (u16.north) {$\{x\}$};

\node[right] at (u29.south) {$\{x\}$};
\node[bdot] (u30) at (30*360/35:4) {};
\node[bdot] (u31) at (31*360/35:4) {};
\node[below] at (u31.west) {$\{z\}$};

\foreach \x in {1,...,4}
\node[dot] (w\x) at (\x*360/40:6) {};
\foreach \x in {7,...,16}
\node[dot] (w\x) at (\x*360/40:6) {};
\foreach \x in {20,...,27}
\node[bdot] (w\x) at (\x*360/40:6) {};
\node[bdot] (w29) at (29*360/40:6) {};
\node[bdot] (w31) at (31*360/40:6) {};
\foreach \x in {35,...,40}
\node[dot] (w\x) at (\x*360/40:6) {};

\path
(v2) edge (u12) (v2) edge (u10) (v2) edge (u11)
(v3) edge (u12)
(v3) edge (u14) (v3) edge (u15) (v3) edge (u16)
(v3) edge (u9) (v1) edge (u9)
(v3) edge (u10) (v1) edge (u10)
(v2) edge (u8) (v1) edge (u8)
(v3) edge (u19) (v4) edge (u19)
(v4) edge (u20) (v4) edge (u21)
(v3) edge (u21) (v6) edge (u21)
(v3) edge (u22) (v6) edge (u22)
(v4) edge (u23) (v6) edge (u23)
(v7) edge (u2) (v1) edge (u2)
(v6) edge (u35) (v7) edge (u35) (v1) edge (u35)
(v1) edge (u4) (v6) edge (u34)
(v1) edge (u6) (v1) edge (u5) (v1) edge (u34)
(v7) edge (u33) (v6) edge (u33)
(v7) edge (u1)
(v6) edge (u30) (v6) edge (u31) (v6) edge (u29)
;

\path
(u31) edge (w35)
(u31) edge (w36) (w36) edge (v7)
(u31) edge (w37) (w37) edge[bend right=15] (v7) (w37) edge[bend right=15] (u4)
(u31) edge (w38) (w38) edge[bend right=5] (v1)
(u31) edge (w39) (w39) edge (u4)
(u31) edge (w40) (w40) edge (u4) (w40) edge[bend right=10] (v1)
(v6) edge[bend right=5] (w1) (w1) edge (u4)
(v6) edge (w2) (w2) edge (u4) (w2) edge (v7)
(v7) edge[bend right=10] (w3) (w3) edge (u4)
(u4) edge (w4)

(u6) edge (w7)
(u6) edge (w8) (w8) edge (v2)
(u6) edge (w9) (w9) edge (v2) (w9) edge[bend right=15] (u14)
(u6) edge (w10) (w10) edge[bend right=15] (v3)
(u6) edge (w11) (w11) edge (u14)
(u6) edge (w12) (w12) edge (u14) (w12) edge (v3)
(v1) edge[bend left=5] (w13) (w13) edge (u14)
(v1) edge[bend left=5] (w14) (w14) edge (u14) (w14) edge (v2)
(v2) edge (w15) (w15) edge (u14)
(u14) edge (w16)

(u16) edge (w20)
(u16) edge (w21) (w21) edge[bend left=5] (v4)
(u16) edge (w22) (w22) edge (v4) (w22) edge[bend right] (u29)
(u16) edge (w23) (w23) edge[bend right=7] (v6)
(u16) edge (w24) (w24) edge[bend right] (u29)
(u16) edge (w25) (w25) edge[bend right] (u29) (w25) edge[bend right=15] (v6)
(v3) edge[bend right=5] (w26) (w26) edge[bend right] (u29)
(v3) edge (w27) (w27) edge[bend right] (u29) (w27) edge[bend right=15] (v4)
(v4) edge (w29) (w29) edge[bend right=15] (u29)
(u29) edge (w31)
;
\end{tikzpicture}
\end{center}
\caption{Concentric Hasse diagram of $\Cm{\alg{F}_{\var{Hi}}}(3)$. 
Shaded bands mark $P_1(3)$, $P_2(3)$ and $P_3(3)$. Solid dots
indicate $S(x)$. Points
without labels are labelled by $\emptyset$.}
\label{fig:P3(3)}    
\end{figure}
Here we have $|\Cm{\alg{F}_{\var{Hi}}}(3)| = 61 = |\Cm{\alg{F}_{\var{Br}}}(3)|$. It is known that
$|\alg{F}_{\var{Br}}(3)| = 623,662,965,552,330$ \textup(cf.~\cite{LW76} or~\cite{Krz77}\textup).
Using the inclusion-exclusion principle, we can calculate
$|\alg{F}_{\var{Hi}}(3)| = |\pw(S(x))\cup\pw(S(y))\cup\pw(S(z))| =
|\pw(S(x))|+|\pw(S(y))|+|\pw(S(z))|-|\pw(S(x)\cap S(y))|
-|\pw(S(y)\cap S(z))|-|\pw(S(z)\cap S(x))|+|\pw(S(x)\cap S(y)\cap S(z))| =
3\cdot 2^{23}-3\cdot 2^3+2 = 25,165,802$. 
\end{example}

\begin{example}
The last pictorial example is $\Cm{\alg{F}_{\var{Ho}}}(2)$ given in
Figure~\ref{fig:Po3(2)}. It is depicted as a cutout of
Figure~\ref{fig:P3(3)} to highlight the fact that
$\Cm{\alg{F}_{\var{Ho}}}(2)$ is an upset in $\Cm{\alg{F}_{\var{Hi}}}(3)$.
\begin{figure}
\begin{center}
\begin{tikzpicture}[scale=0.8]
\clip (-6.4,-2) rectangle (0,4.5);  

\filldraw[black!15!white] (0,0) circle [radius=6.4];
\filldraw[white] (0,0) circle [radius=5.5]; 
  
\filldraw[black!15!white] (0,0) circle [radius=4.6];
\filldraw[white] (0,0) circle [radius=3.4]; 

\filldraw[black!15!white] (0,0) circle [radius=2.3];
\filldraw[white] (0,0) circle [radius=1.1];

\foreach \x in {2,4}
\node[sdot] (v\x) at (\x*360/7:2) {};
\node[dot] (v3) at (3*360/7:2) {};
\node[sdot] (v5) at (5*360/7:1.4) {};
\node[right] at (v4) {$\{x\}$};
\node[right] at (v3) {$\{x,y\}$};
\node[below] at (v2) {$\{y\}$};

\node[sdot] (u11) at (11*360/35:4) {};
\foreach \x in {12,14,15,16,19}
\node[dot] (u\x) at (\x*360/35:4) {};
\node[sdot] (u20) at (20*360/35:4) {};

\node[right] at (u14.north) {$\{y\}$};
\node[dot] (u14) at (14*360/35:4) {};
\node[dot] (u15) at (15*360/35:4) {};
\node[above] at (u16) {$\{x\}$};

\foreach \x in {15,16}
\node[dot] (w\x) at (\x*360/40:6) {};
\foreach \x in {20,21}
\node[dot] (w\x) at (\x*360/40:6) {};

\path
(v2) edge (u12) 
(v2) edge (u11)
(v3) edge (u12)
(v3) edge (u14) (v3) edge (u15) (v3) edge (u16)
(v3) edge (u19) (v4) edge (u19)
(v4) edge (u20) 

(v2) edge (w15) (w15) edge (u14)
(u14) edge (w16)

(u16) edge (w20)
(u16) edge (w21) (w21) edge[bend left=5] (v4)
;
\end{tikzpicture}
\end{center}
\caption{$\Cm{\alg{F}_{\var{Ho}}(2)}$ as an upset of $\Cm{\alg{F}_{\var{Hi}}(3)}$.} 
\label{fig:Po3(2)}    
\end{figure}
\end{example}

\subsection{Varieties of bounded height}

For each $k\in \mathbb{N}$ define a term $d_k(x_1,\dots,x_k)$ inductively by
putting
$d_1(x_1) = x_1$, $d_2(x_1,x_2) = ((x_2\ra x_1)\ra x_2)\ra x_2$,
and 
$d_{k+1}(x_1,\dots, x_{k+1}) = d_2(d_k(x_1,\dots,x_k), x_{k+1})$, for
$k\geq 2$. Let $\var{Hi}_k$ be defined relative to $\var{Hi}$
by the identity
\begin{equation}\label{eq:Hk}\tag{$H_k$}
d_{k+1}(x_1,\dots, x_{k+1}) = 1.
\end{equation}
We call $\var{Hi}_k$ \emph{the variety of height $k$}. The nomenclature is 
justified by the following characterisation. 

\begin{theorem}
Let $\alg{A}\in \var{Hi}$. Then  $\alg{A}\in \var{Hi}_k$ if and only if
the longest chain in $\Cm{\alg{A}}$ is of length at most $k$.  
\end{theorem}

\begin{proof}
First we show that no 
algebra $\alg{A}$ with $h(\alg{A})\leq k$ can falsify~\eqref{eq:Hk}. 
We proceed by induction on $k$. All $\{\ra,1\}$-subreducts of Boolean algebras
satisfy ($H_1$) so the base case holds.
For the inductive step, assume all algebras $\alg{B}$ with
$h(\alg{B}) \leq k$ satisfy~\eqref{eq:Hk}, but some
algebra $\alg{D}$ with $h(\alg{D}) = k+1$ falsifies
$(H_{k+1})$. We can assume $\alg{D} = \alg{A}^\oplus$ with
$h(\alg{A}) = k$. 
Suppose
$$
\left(\bigl(a_{k+2}\ra d_{k+1}(a_1,\dots,a_{k+1})\bigr)\ra
  a_{k+2}\right) \ra a_{k+2}\neq 1
$$
on assigning
$x_i \mapsto a_i\in D$. This implies
$1\neq a_{k+2}\not\leq d_{k+1}(a_1,\dots,a_{k+1})$ and therefore
$d_{k+1}(a_1,\dots,a_{k+1})<\op$. Now, since
$\alg{A} \cong \alg{D}/\theta(\op,1)$, it further follows that
$d_{k+1}(a_1/\mu,\dots,a_{k+1}/\mu)<1$ in $\alg{A}$, where
$\mu=\theta(\op,1)$. This contradicts the inductive hypothesis.

Conversely, assume $\alg{A}$ satisfies~\eqref{eq:Hk}, yet $h(\alg{A})> k$.  
There exist $\mu_1,\dots,\mu_{k+1}\in\cm{\alg{A}}$ with
$\mu_{k+1}\subsetneq \mu_k\subsetneq\dots\subsetneq
\mu_1$. Take $a_1,\dots,a_{k+1}\in A$ such that for each $i\in\{1,\dots,k+1\}$ we
have $\mu_i\vee\theta(a_i,1) = \mu_i^+$, so that
$a_i/\mu_i = \op_{\mu_i}$. Obviously,
$\mu_{k+1}\subsetneq\mu_{k+1}^+\subseteq\mu_k\subsetneq\dots\subsetneq
\mu_2\subsetneq\mu_2^+\subseteq\mu_1\subsetneq \mu_1^+$.
Then $a_1\notin 1/\mu_2$ and so
$d_2(a_1/\mu_2,a_2/\mu_2) = d_2(a_1/\mu_2,\op_{\mu_2}) = 
((\op_{\mu_2}\ra a_1/\mu_2)\ra \op_{\mu_2})\ra\op_{\mu_2}
= \op_{\mu_2}$, establishing
\begin{equation}\label{eq:d2}
d_2(a_1,a_2)/\mu_2 = \op_{\mu_2}.
\end{equation}  
On the other hand, since $\alg{A}$ satisfies~\eqref{eq:Hk}, we
have
\begin{align*}
1/\mu_{k+1}
&= d_{k+1}(a_1/\mu_{k+1},\dots, a_{k+1}/\mu_{k+1})\\
&= d_2(d_{k}(a_1/\mu_{k+1},\dots,a_{k}/\mu_{k+1}),a_{k+1}/\mu_{k+1})\\
&= d_2(d_{k}(a_1/\mu_{k+1},\dots,a_{k}/\mu_{k+1}),\op_{\mu_{k+1}})
\end{align*}
and therefore $d_{k}(a_1,\dots,a_{k})/\mu_{k+1} = 1/\mu_{k+1}$. Hence,
$d_{k}(a_1,\dots,a_{k})/\mu_{k} = 1/\mu_{k}$ and repeating the
reasoning $k-2$ times we obtain $d_2(a_1,a_2)/\mu_2 = 1/\mu_2$,
contradicting~\eqref{eq:d2}.
\end{proof}

Consider any subvariety $\mathcal{V}$ of $\var{Hi}$
or of $\var{Ho}$. 
Let $\mathcal{V}_k\deq \mathcal{V}\cap \mathrm{Mod}(H_k)$. 
The structure of the $n$-generated free algebra
$\alg{F}_{\mathcal{V}_k}(n)$ is the same as $\FV(n)$ except that
$\Cm{\alg{F}_{\mathcal{V}_k}(n)} =
\left(\bigcup_{r=1}^{k} P_r(n);\subseteq\right)$,
where $P_r(n)$ are computed in $\mathcal{V}$. 
Note that for any 
$G\in \mathrm{Up}\left(\bigcup_{r=1}^{k}P_r(n)\right)$ 
such that $G\cap P_{k}(n)\neq \emptyset$ we have 
$\bigl(\alg{F}_{\mathcal{V}_k}(n)/\bigcap G\bigr)^\oplus\notin\mathcal{V}_k$
and so summing up to $k$ suffices (see the remarks 
preceding Lemma~\ref{lem:layers-gen}).
Applying the map $S$ from Theorem~\ref{thm:mono}~\eqref{eq:S} to
$\Cm{\alg{F}_{\mathcal{V}_k}(n)}$ we immediately obtain the next result. 

\begin{theorem}\label{thm:free-Hik}
Let $\mathcal{V}$ be as above. For any $n > 0$, we have:
$$
\alg{F}_{\mathcal{V}_k}(n)\cong
\left(\bigcup_{i=1}^n\pw(S(x_i));\ra,\emptyset\right)
\leq \alg{Ac}(\Cm{\alg{F}_{\mathcal{V}_k}(n)}),
$$
for $\mathcal{V}_k\subseteq \var{Hi}$;
in particular, $\alg{F}_{\mathcal{V}_k}(n)\cong\FV(n)$
for any $k\geq n$.   
$$
\alg{F}_{\mathcal{V}_k}(n)\cong
\left(\bigcup_{i=1}^{n}\pw(S(x_i))\cup\pw(S(0));\ra,\emptyset, P_1(n)\right)
\leq \alg{Ac}(\Cm{\alg{F}_{\mathcal{V}_k}(n)}),
$$
for $\mathcal{V}_k\subseteq \var{Ho}$, where 
$P_1(n)$ is the interpretation of\/ $0$;
in particular,
$\alg{F}_{\mathcal{V}_k}(n)\cong\FV(n)$ for any $k\geq n+1$.   
\end{theorem}

Analogously, consider a subvariety $\mathcal{W}$ of $\var{Br}$ or of
$\var{Bo}$ and let
$\mathcal{W}_k\deq \mathcal{W}\cap \mathrm{Mod}(B_k)$. 

\begin{theorem}\label{thm:free-Brk}
Let $\mathcal{W}$ be as above. For any $n > 0$, we have
$$
\alg{F}_{\mathcal{W}_k}(n)\cong
\left\langle\bigcup_{i=1}^n\pw(S(x_i));\ra,\emptyset\right\rangle
\cong \alg{Acm}(\Cm{\alg{F}_{\mathcal{W}_k}(n)}),
$$
for $\mathcal{W}_k\subseteq \var{Br}$; in particular,
$\alg{F}_{\mathcal{W}_k}(n)\cong\alg{F}_{\mathcal{W}}(n)$ for any
$k\geq n$.   
$$
\alg{F}_{\mathcal{W}_k}(n)\cong
\left\langle\bigcup_{i=1}^{n}\pw(S(x_i))\cup\pw(S(0));\ra,\emptyset, P_1(n)\right\rangle
\cong \alg{Acm}(\Cm{\alg{F}_{\mathcal{W}_k}(n)}),
$$
for $\mathcal{W}_k\subseteq \var{Bo}$;
in particular, $\alg{F}_{\mathcal{W}_k}(n)\cong\alg{F}_{\mathcal{W}}(n)$ for any
$k\geq n+1$.   
\end{theorem}

\subsection{Varieties of bounded width}
We begin by defining a shorthand. 
For a set of terms $T = \{t_1,\dots, t_\ell\}\cup\{s\}$ we will write
$T\ra s$ or $\{t_1,\dots, t_\ell\}\ra s$ for the term
$t_1\ra(t_2\ra \dots \ra (t_\ell\ra s)\dots)$.
By equation (6) of Section~\ref{sec:Hi}, this
notation is unambiguous, since
$t_1\ra(t_2\ra \dots \ra (t_\ell\ra s)\dots) = t_{\pi(1)}\ra(t_{\pi(2)}\ra \dots \ra
(t_{\pi(\ell)}\ra s)\dots)$ holds for any permutation $\pi$ of
$\{1,\dots,\ell\}$. 
Now, for each $k\in\mathbb{N}$
pick a set of pairwise distinct variables
$\bigl\{x_i,: i\in\{1,\dots,k+1\}\bigr\}\cup\{y\}$.
For each $i\in\{1,\dots,k+1\}$ define a term
$$
q_{k,i}(x_1,\dots,x_{k+1}) \deq \{x_j: j\neq i\}\ra x_i
$$
and then put
\begin{equation}\label{eq:Bk}\tag{$B_k$}
\{q_{k,i}(x_1,\dots,x_{k+1})\ra y: 1\leq i\leq k+1\}\ra y = 1.
\end{equation}
Let $\var{Hb}_k$ be the variety defined relative to $\var{Hi}$
by~\eqref{eq:Bk}. 

\begin{theorem}
Let $\alg{A}\in \var{Hi}$. Then, $\alg{A}\in\var{Hb}_k$ if and only if
the longest antichain in $\Cm{\alg{A}}$ with a common lower bound, is of length
at most $k$.   
\end{theorem}

\begin{proof}
Let $w(\alg{A})$ stand for the length of the longest
antichain in $\Cm{\alg{A}}$ with a common lower bound.
First we show that every 
algebra $\alg{A}$ with $w(\alg{A})\leq k$ satisfies~\eqref{eq:Bk}. 
Suppose $w(\alg{A}) \leq k$, but $\alg{A}$ falsifies~\eqref{eq:Bk}.
Since taking quotients cannot increase the width, we can assume 
$\alg{A}$ is subdirectly irreducible. Take elements
$a_1,\dots,a_{k+1},c\in A$ such that
\begin{equation}\label{eq:fail-q}\tag{\dag}
\alg{A}\models
\{q_{k,i}(a_1,\dots,a_{k+1})\ra c:1\leq i\leq k+1\}\ra c\neq 1.
\end{equation}
We can also assume $c = \op$ (otherwise just take a suitable quotient).
Hence, $q_{k,i}(a_1,\dots,a_{k+1})\ra \op = 1$
for all $i\in\{1,\dots,k+1\}$ and thus $q_{k,i}(a_1,\dots,a_{k+1})\leq\op$
for every $i\in\{1,\dots,k+1\}$.

Next, we claim that the elements $a_i$ and $a_j$ are incomparable, for $i\neq j$.
To see it, suppose that $a_{i_0}\leq a_{j_0}$ for some
distinct $i_0$ and $j_0$. Then
$q_{k,{j_0}}(a_1,\dots,a_{k+1}) = 1$, so 
$q_{k,{j_0}}(a_1,\dots,a_{k+1})\ra\op = \op$, and thus
$$
\alg{A}\models
\{q_{k,i}(a_1,\dots,a_{k+1})\ra \op:1\leq i\leq k+1\}\ra \op = 1,
$$
contradicting~\eqref{eq:fail-q}.
We conclude that $\{a_1,\dots,a_{k+1}\}$ is an antichain.
Therefore, by Theorem~\ref{thm:PrimeFilter},
for each $i\in\{1,\dots,k+1\}$ there exists
$\mu_i\in \Cm{\alg{A}}$ such that $a_i\in 1/\mu_i$ and
$a_j\notin 1/\mu_i$ for all $j\neq i$. Hence,
$\{\mu_i: 1\leq i\leq k+1\}\subseteq\Cm{\alg{A}}$ is an antichain,
and since $\alg{A}$ is subdirectly irreducible, the identity congruence
belongs to $\Cm{\alg{A}}$, so $w(\alg{A}) \deq k+1$ contradicting the supposition.

Conversely, assume $\alg{A}$ satisfies~\eqref{eq:Bk}, yet $w(\alg{A})> k$.  
Assume without loss that $\alg{A}$ is subdirectly irreducible,
and $w(\alg{A}) = k+1$. Let $\mu_1,\dots,\mu_{k+1}\in\cm{\alg{A}}$ be an
antichain with a common lower bound. By CD we get
$\bigcap\{\mu_j: 1\leq j\leq k+1, \ j\neq i\} \not\subseteq \mu_i$,
for every $i\in \{1,\dots,k+1\}$. Therefore,
for each $i$ there is $a_i\in
\bigcap\{\mu_j: 1\leq j\leq k+1, \ j\neq i\}\smallsetminus \mu_i$. We claim that 
$a_1,\dots,a_{k+1}$ is an antichain. To see that,
suppose $a_i\ra a_s = 1$ for some $i\neq s$.
Then $a_i$ and $a_i\ra a_s$ both belong to
$\bigcap\{\mu_j: 1\leq j\leq k+1, \ j\neq i\}$, and so
$a_s\in \bigcap\{\mu_j: 1\leq j\leq k+1, \ j\neq i\}\subseteq \mu_s$.
Hence $a_s\in 1/\mu_s$ contradicting the choice of $a_s$.
Let $\alg{B}$ be the subalgebra of $\alg{A}$  generated by
$\{\op,a_1,\dots,a_{k+1}\}$. Clearly, $\alg{B}$ is finite and subdirectly
irreducible. Now, using Theorem~\ref{thm:PrimeFilter} 
we find $\nu_1,\dots,\nu_{k+1}\in \cm{\alg{B}}$ such that
$\{a_1,\dots,a_{k+1}\}\cap 1/\nu_i = \{a_i\}$. 
Hence $\nu_1,\dots,\nu_{k+1}$ is an antichain. 
We conclude that $\alg{B}$ is finite subdirectly irreducible and
$w(\alg{B}) = k+1$. 

By congruence distributivity and finiteness we have that
$\J{\Con{\alg{B}}}$ and $\Cm{\alg{B}}$ are isomorphic.
Hence, the greatest length of any antichain in $\J{\Con{\alg{B}}}$ is 
$k+1$, and by Theorem~\ref{thm:irreducible-charact} any such antichain of length
$k+1$ in $\J{\Con{\alg{B}}}$ consists of congruences 
$\theta(1,b_1),\dots,\theta(1,b_{k+1})$ such that 
$\{b_i: 1\leq i\leq k+1\}$ is an antichain of irreducible elements. Therefore
$q_{k,i}(b_1,\dots,b_{k+1}) = b_i<\op$, so 
$q_{k,i}(b_1,\dots,b_{k+1})\ra \op = 1$ for each $i\in\{1,\dots,k+1\}$. Hence
$$
\{q_{k,i}(b_1,\dots,b_{k+1})\ra \op:1\leq i\leq k+1\}\ra \op = \op,
$$
contradicting the assumption that $\alg{A}$ satisfies~\eqref{eq:Bk}.
\end{proof}

Consider any subvariety $\mathcal{V}$ of $\var{Hi}$ or
of $\var{Ho}$. 
Let $\mathcal{R}_k\deq \mathcal{V}\cap \mathrm{Mod}(B_k)$. 
Construction of the poset $\Cm{\alg{F}_{\mathcal{R}_k}(n)}$ for the $n$-generated free
$\mathcal{R}_k$-algebra proceeds as follows.

\begin{lemma}\label{lem:bounded-width}
Let $k\geq 1$. We build $P_1(n)$ exactly as for $\alg{F}_{\var{Hi}}(n)$
or for $\alg{F}_{\var{Ho}}(n)$, depending on whether\/ $0$ is in the signature.   
For $k\geq 2$, and for any $G\in \mathrm{Up}\left(\bigcup_{r=1}^{k-1}P_r(n)\right)$
such that any antichain in $G$ has at most $k$ elements and
$G\cap P_{k-1}(n)\neq \emptyset$, and for any $L$ such that
$L\subsetneq X\cap 1/\bigcap G$, 
let $f^G_L\colon X \to  \bigl(\FV(n)/\bigcap G\bigr)^\oplus$ be the map
given by 
$$
f^G_L(x) =\begin{cases}
  1 & \text{ if } x\in L,\\
  \op & \text{ if } x\in \left(1/\bigcap G\right)\smallsetminus L,\\
  x/\bigcap G & \text{ if } x\notin 1/\bigcap G.
\end{cases} 
$$
Let $\ov{f}^G_L$ be the homomorphism uniquely extending $f^G_L$ and
let $\mu^G_L = \ker\ov{f}^G_L$. Then, writing
$w(G)\leq k$ to abbreviate ``any antichain in $G$ has at most $k$ elements and
$G\cap P_{k-1}(n)\neq \emptyset$'', we obtain
$$
P_k(n) =
\left\{\mu^G_L: G\in \mathrm{Up}\left(\bigcup_{r=1}^{k-1}P_r(n)\right),\ 
w(G)\leq k,\ L\subsetneq X\cap 1/\bigcap G  
\right\}.
$$
\end{lemma}  

Having constructed $\Cm{\alg{F}_{\mathcal{R}_k}(n)}$, we simply
apply Theorem~\ref{thm:free-Hilbert-alg} as it stands. Similarly,
Theorem~\ref{thm:free-Br-semi} applies to varieties of Brouwerian semilattices
of bounded width. 

\subsubsection{The linear case: free spectra}\label{ssec:width-one}
 
The narrowest nontrivial case, $k=1$, deserves a separate treatment, since we
can get a closed formula for the free spectrum. Recall that 
$\alg{A}\in\var{Hb}_1$ if and only if
$\alg{A}\models
((x_1\ra x_2)\ra y) \ra \bigl(((x_2\ra x_1)\ra y) \ra y\bigr) = 1$
if and only if every antichain in $\cm{\alg{A}}$ with a
common lower 
bound is a singleton, that is, $\cm{\alg{A}}$ is a forest of chains. On the
logic side the variety $\var{Hb}_1$ corresponds to the G\"odel--Dummett logic
$\mathsf{LC}$.
Let $a(n)$ be the $n$-th ordered Bell (or Fubini)
number: $a(n)=\sum_{j}\mathrm{Surj}(n,j)$, where
$\mathrm{Surj}(n,j)=j!\begin{Bmatrix}n\\ j\end{Bmatrix}$
counts surjections from an $n$-element set onto a $j$-element set
and $\begin{Bmatrix}n\\ j\end{Bmatrix}$ is
a Stirling number of the second kind.

As already noted in Słomczyńska~\cite{Slo05}, the poset of completely
meet-irreducible congruences is \emph{the same} for the linear varieties of
equivalential algebras, Hilbert algebras and Brouwerian semilattices.
The structure of $\cm{\alg{F}_{\var{Hb}_1}(n)}$ 
was determined\footnote{In the context of
  equivalential algebras.} in~\cite{Slo05}. Here is the
main fact we will need. 
\begin{equation}\label{eq:w1-poset}\tag{\ddag}
|\cm{\alg{F}_{\var{Hb}_1}(n)}| =
\sum_{k=1}^{n} k!\begin{Bmatrix}n+1\\ k+1\end{Bmatrix} = 2a(n)-1.
\end{equation}
We do not reproduce the argument; we will focus instead on the
free spectrum of $\var{Hb_1}$. Note that the free spectrum
is different for equivalential algebras, Brouwerian semilattices and Hilbert
algebras, in spite of the underlying poset being the same, so our results
here are independent of the results from~\cite{Slo05}.
For $\var{Hi}$, in particular for $\var{Hb_1}$, the free spectrum is governed
by antichain selection function $S(x)$.   
Recall from Theorem~\ref{thm:free-Hilbert-alg} that the universe of 
$\alg{F}_{\var{Hb}_1}(n)$ is $\bigcup_{i=1}^{n}\pw(S(x_i))$,
where
\[
  S(x)=\{\mu\in \cm\alg{F}_{\var{Hb}_1}(n):
  x\notin 1/\mu \text{ and } x\in 1/\nu \text{ for all }
  \nu\supsetneq\mu\}. 
\] 
Thus $\mu\in S(x)$ if and only if $x$ is identified with $1$ by all congruences
strictly above $\mu$ but not by $\mu$; equivalently, $x/\mu = \op$.

\begin{lemma}\label{lem:width1-S}
Consider $\alg{F}_{\var{Hb_1}}(n)$ with the set $X$ of $n\geq 2$ free generators.
For any distinct $x_{i_1},\dots, x_{i_r}\in X$ we have the following
\begin{enumerate}
\item if $r<n$, then $\bigl|\bigcap_{t=1}^{r}S(x_{i_t})\bigr|= 4a(n-r)-1$,
\item if $r = n$, then $\bigl|\bigcap_{t=1}^{r}S(x_{i_t})\bigr|= 1$.
\end{enumerate}  
\end{lemma}

\begin{proof}
For (1), let $x_{i_1},\dots, x_{i_r}\in X$ be distinct.  
Observe that $\bigl|\bigcap_{t=1}^{r}S(x_{i_t})\bigr| =
\sum_{h=1}^n\bigl|\bigcap_{t=1}^{r}S(x_{i_t})\cap P_h(n)\bigr|$.
Fix $h$, and recall that
an element $\mu\in\bigcap_{t=1}^{r}S(x_{i_t})\cap P_h(n)$ is the kernel of
a surjective homomorphism $\overline{f}$ extending the map $f$ from $X$ to
a linearly ordered algebra $\alg{L}\in\var{Hi_1}$ with
$|L| = h+1$ such that
$f(x_{i_1}) = \dots = f(x_{i_r}) = \op$ and 
$L\smallsetminus\{1,\op\}\subseteq f(X)$.
This leads to three possibilities: (i) $f$ is a surjection from
$X\smallsetminus\{x_{i_1},\dots,x_{i_r}\}$ onto a $(h-1)$-element chain
(that is, $f^{-1}(\{1,\op\})=\{x_{i_1},\dots,x_{i_r}\}$), (ii) $f$ is a surjection onto
a chain of length $h$ (that is, $f^{-1}(1)=\emptyset$ or
$f^{-1}(\op)=\{x_{i_1},\dots,x_{i_r}\}$),
(iii) $f|_{X\smallsetminus\{x_{i_1},\dots,x_{i_r}\}}$ is a surjection onto a chain of length
$h+1$. Thus, we get 
\begin{align*}
\bigl|\bigcap_{t=1}^{r}S(x_{i_t})\cap P_h(n)\bigr|
&= \mathrm{Surj}(n-r,h-1)+ 2\,\mathrm{Surj}(n-r,h) + 
\mathrm{Surj}(n-r,h+1)\\
&= (h-1)!\begin{Bmatrix}n-r\\ h-1\end{Bmatrix} +
2(h!)\begin{Bmatrix}n-r\\ h\end{Bmatrix} +
  (h+1)!\begin{Bmatrix}n-r\\ h+1\end{Bmatrix}
\end{align*}
and hence, using the recursive formula for Stirling numbers and the definition
of Fubini numbers we obtain
\begin{align*}
\bigl|\bigcap_{t=1}^{r}S(x_{i_t})\cap P_h(n)\bigr|
 &=
\sum_{h=1}^n(h-1)!\left(\begin{Bmatrix}n-r+2\\ h+1\end{Bmatrix} -
   \begin{Bmatrix}n-r+1\\ h+1\end{Bmatrix}\right)\\
 &= 4 a(n-r) - 1.
\end{align*}
Now, (2) is clear since $\bigcap_{t=1}^{n}S(x_{i_t})$ is a singleton. 
\end{proof}

Now we are ready to calculate the free spectrum of $\var{Hb_1}$. 

\begin{theorem}\label{thm:width1-spectrum}
For every $n\geq1$,
\[
  |\alg{F}_{\var{Hb}_1}(n)|
  = (-1)^{\,n-1}\left[\,2
    + \frac12\sum_{m=1}^{n-1}(-1)^{\,m}\binom{n}{m}\,16^{a(m)}\right].
\]
\end{theorem}

\begin{proof}
By Theorem~\ref{thm:free-Hilbert-alg} the underlying set is
$\bigcup_{i=1}^{n}\pw(S(x_i))$. Inclusion--exclusion with
$\pw(A)\cap\pw(B)=\pw(A\cap B)$ gives
\[
  |\alg{F}_{\var{Hb}_1}(n)|
  =\sum_{r=1}^{n}(-1)^{r-1}\binom{n}{r}
\pw(S(x_1)\cap\dots\cap S(x_r)).
\]
Consequently, by Lemma~\ref{lem:width1-S}
$$
|\alg{F}_{\var{Hb}_1}(n)| = 
2(-1)^{n-1}+ \frac{1}{2}\sum_{r=1}^{n-1}(-1)^{r-1}\binom{n}{r}\,16^{a(n-r)}
$$
Substituting $m=n-r$, we immediately obtain
the final formula.
\end{proof}

\begin{example}\label{ex:width1-values}
\[
\renewcommand{\arraystretch}{1.2}
\begin{array}{c|c|l}
 n & |\cm{\alg{F}_{\var{Hb}_1}(n)}| & |\alg{F}_{\var{Hb}_1}(n)| \\\hline
 1 & 1    & 2 \\
 2 & 5    & 14 \\
 3 & 25   & 6\,122 \\
 4 & 149  & 9\,007\,199\,254\,728\,734 \\
 5 & 1081 & \approx 5.09\times10^{90} \\
 6 & 9365 & > 10^{652}
\end{array}
\]
For $n=2$ the algebra still has width one, so
$|\alg{F}_{\var{Hb}_1}(2)|=14=|\alg{F}_{\var{Hi}}(2)|$ of
Example~\ref{ex:free-Hi2}. The value $|\alg{F}_{\var{Hb}_1}(4)|$ sits just below
$2^{53}$.
\end{example}




\subsection{Structural completeness}
Structural completeness was first isolated as a property of a logic
in~\cite{Pog71}. Algebraically (cf.~\cite{Ber91}), a variety $\mathcal{K}$ is 
structurally complete if all quasi-equations true in the free 
$\mathcal{K}$-algebra $\alg{F}_\mathcal{K}(\omega)$ are true in $\mathcal{K}$,
and $\mathcal{K}$ is \emph{hereditarily structurally complete}  
if all its subvarieties are structurally complete.

It is well known that $\var{Hi}$, $\var{Br}$, and $\var{Bo}$,
are hereditarily structurally complete. For $\var{Hi}$
it was proved by Prucnal~\cite{Pru72} by a syntactic method, involving a clever
substitution that is now commonly called \emph{Prucnal's trick}. This result
was extended to $\var{Br}$ and $\var{Bo}$ by  Minari and Wroński~\cite{MW88}.
In contrast, $\var{Ho}$ is not structurally
complete, as shown by Wro\'nski~\cite{Wro86}.

Recall that for any terms $t_1, \dots, t_k$ and $r$ in
$\mathrm{Term}_{\{\ra,0\}}(n)$ or $\mathrm{Term}_{\{\ra,\wedge,0\}}(n)$,
we write $\{t_1,\dots, t_k\}\ra r$ for 
$t_{1}\ra(t_{2}\ra\dots \ra(t_{k}\ra r)\dots)$. We further
define $\neg t \deq t\ra 0$, and write 
$\neg\{t_1,\dots, t_k\}$ for $\{t_1,\dots, t_k\}\ra 0$.

Cintula and Metcalfe~\cite{CM10} found a quasi-equational base
for quasi-equations satisfied by the free algebra $\alg{F}_{\mathcal{V}}(\omega)$
in any subvariety $\mathcal{V}$ of $\var{Ho}$.  
It consists of quasi-identities 
\begin{equation}\label{eq:quasi-id}\tag{$W_n'$}
\neg\{x_1, x_2, \dots, x_n\} = 1 \doublewedge\bigdoublewedge_{i=1}^n
(\neg\neg x_i\ra x_i)\ra x_{n+1} = 1  \implies x_{n+1} = 1
\end{equation}  
for each $n<\omega$. By convention we take
($W_0'$) to be $x_{1} = 1\implies x_{1} = 1$;
next ($W_1'$) is equivalent to $x_2 = 1 \implies x_2 = 1$, which again holds
trivially. The first nontrivial case is 
the original Wro\'nski's quasi-equation ($W_2'$).
For each $n<\omega$ define the poset
$\mathbb{W}_n$ putting
$\mathbb{W}_n = \{b_i:1\leq i\leq n\}\cup\{a_i:1\leq i\leq n\}\cup\{\bot\}$
and ordering it by the reflexive transitive closure
of $\{(\bot, a_i): i\leq n\}\cup \{(a_i,b_j): i\neq j\}$.
Then by direct verification we get that
($W_m'$) holds in the algebra $\alg{Ac}(\mathbb{W}_n)$ if and only if
$m<n$. The reader may compare our $\mathbb{W}_n$ to Lemma~5.1 of~\cite{CM10}.

\begin{cor}
Let $\mathcal{V}\subseteq\var{Ho}$. If $\alg{Ac}(\mathbb{W}_n)\in \mathcal{V}$
for some $n\geq 2$, then $\mathcal{V}$ is not structurally complete.  
\end{cor}  

Our construction of the free algebras lets us immediately show that
all they satisfy~\eqref{eq:quasi-id}.

\begin{lemma}\label{lem:1-holds-on-free}
Let $\mathcal{V}\subseteq \mathsf{Ho}$. Then~\eqref{eq:quasi-id} holds in
$\FV(m)$ for any $n, m<\omega$. 
\end{lemma}  

\begin{proof}
Consider a substitution $\sigma$ extending the map
$x_i\mapsto t_i$ for some $t_i\in\mathrm{Term}_{\{\ra,0\}}$, 
such that the antecedent of~\eqref{eq:quasi-id}  
holds in $\FV(m)$. Working in the representation from
Theorem~\ref{thm:free-Hi0-alg}, we calculate: 
\begin{align*}
\emptyset = S(1)
&= S(\neg\{t_1,\dots, t_n\})
   = S(t_1)\ra \dots\ra\bigl(S(t_n)\ra P_1(m)\bigr)\\
&= \bigl(\dots\bigl(P_1(m)\smallsetminus
  S(t_n)\bigr)\smallsetminus\dots\bigr)\smallsetminus S(t_1)
= P_1(m)\smallsetminus \bigcup_{i=1}^nS(t_i).  
\end{align*}
Hence $\bigcup_{i=1}^nS(t_i)\supseteq P_1(m)$. We claim that
$S(t_i)\subseteq P_1(m)$ for some $i\leq n$. For suppose the contrary.
Then, $\bigcup_{i=1}^nS(t_i)\subseteq \bigcup_{i=1}^{m}S(x_i)$,
and since $\mu_X\in P_1(m)\smallsetminus \bigcup_{i=1}^mS(x_i)$,
we obtain $P_1(m)\smallsetminus \bigcup_{i=1}^nS(t_i)\supseteq
P_1(m)\smallsetminus \bigcup_{i=1}^mS(x_i) \neq \emptyset$, a contradiction. 
Without loss of
generality, assume $S(t_1)\subseteq P_1(m)$. Then, straightforward calculations yield
$$
S(\neg\neg t_1\ra t_1) = S(t_1)\smallsetminus P_1(m) = \emptyset  
$$
and therefore 
$S\bigl((\neg\neg t_1\ra t_1)\ra t_{n+1}\bigr) =
\emptyset \ra S(t_{n+1}) = S(t_{n+1})$. 
By assumptions, however,
$S\bigl((\neg\neg t_1\ra t_1)\ra t_{n+1}\bigr) = \emptyset$,
so $S(t_{n+1}) = \emptyset$ as required.
\end{proof}

An analogue of Lemma~\ref{lem:1-holds-on-free} fails for
$\mathsf{Bo}$. The proof above breaks down, 
since it is not the case that $\bigcup_{i=1}^nS(t_i)\supseteq P_1(m)$
implies
$S(t_i)\subseteq P_1(m)$ for some $i\leq n$.

There is a
tight connection between structural completeness and \emph{projectivity},
discovered in Ghilardi~\cite{Ghi99}, of which we only need to recall the fact
that projectivity of algebras in a variety $\mathcal{V}$ implies hereditary
structural completeness of $\mathcal{V}$. Another consequence
of projectivity in our context is the existence of a natural translation of
quasi-equations to equations. We will approach it employing some tools 
from~\cite{Slo12}, which we will now recall and slightly modify or
extend.

\begin{lemma}[\cite{Slo12} Proposition 2.2]\label{lem:Prop2.6}
Let $\mathcal{V}$ be any variety, let $\alg{A}\in\mathcal{V}$ be projective
\textup(in particular, $\alg{A}$ can be free in $\mathcal{V}$\textup) and let
$\psi\in\Con{\alg{A}}$. The following are equivalent:
\begin{enumerate}
\item $\alg{A}/\psi$ is projective.
\item There exists an endomorphism $\tau\colon \alg{A}\to\alg{A}$ such that
  $\ker(\tau) = \psi$ and for every $a\in A$
  we have $\tau(a)\equiv_\psi a$.  
\end{enumerate}  
\end{lemma}

\begin{lemma}\label{lem:criteria}
Let $\mathcal{V}$ be a Fregean, subtractive variety of finite signature.
Assume moreover that for any subdirectly irreducible $\alg{A}\in\mathcal{V}$
with $|A|> 2$, the set $A\setminus\{\op\}$ is a subuniverse of $\alg{A}$.  
Let $\psi\in\Con\FV(n)$. If for each term $t\in 1/\psi$ we have
$t(\overline{1}) = 1$, where $\overline{1}$ is an appropriate constant sequence
with all terms equal to $1$, then $\FV(n)/\psi$ is projective.
\end{lemma}  

\begin{proof}
By Theorem~4.1 of~\cite{Slo12} we know that $\mathcal{V}$ is locally finite. 
So, there exist terms $t_1,\dots, t_k\in 1/\psi$ such that
$\psi = \bigvee_{i=1}^k\theta(1,t_i)$. We proceed by induction on $k$.
For $k=1$ the claim is precisely Lemma~5.1 of~\cite{Slo12}. For $k\geq 2$,
putting $\varphi = \bigvee_{i=1}^{k-1}\theta(1,t_i)$ we get by the inductive
hypothesis that $\FV(n)/\varphi$ is projective.
Then, by Lemma~\ref{lem:Prop2.6}, there exists an endomorphism $\tau$ of
$\FV(n)$ such that for every $i\in\{1,\dots,n\}$ 
we have $\tau(x_i)\equiv_{\varphi} x_i$ and $\ker(\tau) = \varphi$. 
Hence, $\tau(t_k)\equiv_{\varphi} t_k$, and since $\varphi\subseteq\psi$ we get
$t_k\equiv_\psi 1$, and therefore $\tau(t_k)\in 1/\psi$.
But $\tau(t_k)(\overline{1}) = 1$ by our assumptions on $t_k$, so
applying Lemma~5.1 of~\cite{Slo12} we conclude
$\FV(n)/\theta\bigl(\tau(t_k),1\bigr)$ is projective.
Hence, there exists
an endomorphism $\gamma$ of $\FV(n)$ such that 
for every $i\in\{1,\dots,n\}$
we have $\gamma(x_i)\equiv_{\theta\bigl(\tau(t_k),1\bigr)} x_i$ and
$\gamma(\tau(t_k)) = 1$.
Now, observe that
$\gamma(\tau(x_i))\equiv_{\theta(\tau(t_k),1)}\tau(x_i)
\equiv_{\psi} x_i$ holds for every $x_i$ and therefore, since
$\theta(\tau(t_k),1)\subseteq\psi$, we have that 
$(\gamma\circ\tau)(x_i)\equiv_\psi x_i$ for every $x_i$.  
As moreover $(\gamma\circ\tau)(t_i) = 1$ for every $i\in\{1,\dots,k\}$, we get that
$\psi = \ker(\gamma\circ\tau)$. Hence, $\FV(n)/\psi$ is projective, as claimed.
\end{proof}  

The results from~\cite{Pru72, MW88} follow as corollaries. 

\begin{cor}\label{cor:SC}
The varieties $\var{Hi}$, $\var{Br}$ and $\var{Bo}$
are hereditarily structurally complete.
\end{cor}  

\begin{proof}
First we note that all quotients of free algebras in
$\var{Hi}$, $\var{Br}$ are projective. This
follows immediately from 
Lemma~\ref{lem:criteria}, since all terms $t$ satisfy
$t(\overline{1}) = 1$.
For $\var{Bo}$, by Theorem~4.7 and Proposition~4.11 of~\cite{Slo12}
for any variety $\mathcal{V}\subseteq \var{Bo}$ we have that
$\FV(n)/\psi$ is projective for any congruence $\psi$ such that
$0\notin 1/\psi$. Therefore all non-trivial quotients of
$\FV(n)$ are projective. Hereditary structural completeness   
follows.
\end{proof}

The next result gives a sufficient condition for a quasi-equation to be equivalent
to an equation. 

\begin{theorem}\label{thm:translation}
Let $\mathcal{V}\subseteq \var{Ho}$ and let
$t_1,\dots,t_k,r\in \mathrm{Term}_{\{\ra,0,1\}}(n)$.   
If for every $i\in\{1,\dots,k\}$ we have $t_i(\overline{1}) = 1$, then
$$
\FV(n)\models \bigdoublewedge_{i=1}^k t_i = 1 \implies r = 1
\quad\text{if and only if}\quad
\FV(n)\models \{t_1,\dots, t_k\}\ra r = 1.
$$
\end{theorem}  

\begin{proof}
Only the forward direction is not obvious, so suppose
$\FV(n)\models \bigdoublewedge_{i=1}^k t_i = 1 \implies r = 1$. 
Let $\varphi = \bigvee_{i=1}^k\theta(t_i,1)$, and let
$s\in 1/\varphi$. Then $\{t_1,\dots,t_k\}\ra s = 1$, and hence
$\bigl(t_1(\overline{1})\cdots t_k(\overline{1})\bigr)\ra s(\overline{1}) = 1$. But by
assumption $t_i(\overline{1}) = 1$ for each $i$, so we get
$1 = \{t_1(\overline{1}),\dots, t_k(\overline{1})\}\ra s(\overline{1}) =
1\ra s(\overline{1}) = s(\overline{1})$. This shows that $\varphi$ satisfies the 
conditions of Lemma~\ref{lem:criteria}, so
$\FV(n)/\varphi$ is projective. Now by Lemma~\ref{lem:Prop2.6}  
there exists an endomorphism $\tau\colon \FV(n)\to \FV(n)$ such that
$\tau(x_i)\equiv_\varphi x_i$ for each $i\in\{1,\dots,n\}$ and
$\ker(\tau) = \varphi$. Therefore $\tau(t_i) = 1$ for each
$i\in\{1,\dots,k\}$, and hence $1 = \tau(r) =
r(\tau(x_1),\dots,\tau(x_n))\equiv_\varphi r(x_1,\dots,x_n) = r$. 
It follows that $r\in 1/\varphi$ and this in turn means, by definition of
$\varphi$, that $r$ belongs to the filter generated by
$\{t_1,\dots,t_k\}$ and therefore
$\{t_1,\dots,t_k\}\ra r = 1$ as required.
\end{proof}  

Now we characterise projective homomorphic images of the free
algebras in $\mathsf{Ho}$. 

\begin{lemma}\label{lem:unique-term}
Let $\mathcal{V}\subseteq \var{Ho}$. Then, there exists precisely one term
$\ell\in\Irr{\FV(n)}$ such that $\ell(\overline{1})=0$. Namely,
$$
\ell(x_1,\dots,x_n) = \neg\{x_1,\dots, x_n\}.
$$
Therefore, for any
$\psi\in\Con{\FV(n)}$ such that $\ell\notin 1/\psi$, the algebra
$\FV(n)/\psi$ is projective.
\end{lemma}

\begin{proof}
Recall the map $S$ defined in Theorem~\ref{thm:mono}. Restricting $S$ to
$\Irr{\FV(n)}$, as in Theorem~\ref{thm:repres}(1), we obtain a one-to-one
correspondence with $\{\{\mu\}: \mu\in\Cm{\FV(n)}\}$.
Take $\ell$ such that $S(\ell) = \{\mu_X\}$, where $X = \{x_1,\dots,x_n\}$,
cf. Lemma~\ref{lem:P1-with-0}. It easily follows from the definition
of $S$ that $\ell$ has the required form. 

Now let $\psi\in\Con{\FV(n)}$ be such that $\ell\notin 1/\psi$.
Then 
$\psi = \bigvee_{i=1}^{k}\theta(1,u_i)$, for some $u_i$ with
each $u_i\in\Irr{\FV(n)}\smallsetminus\{\ell\}$. 
By the first paragraph, $u_i(\overline{1}) = 1$ for each $i$. Take
$t\in 1/\psi$. Then, $\{u_1,\dots, u_k\}\ra t = 1$, 
and consequently, substituting $1$ for every $x_i$ we get
$\bigl\{u_1(\overline{1}),\dots, u_k(\overline{1})\bigr\}\ra t(\overline{1}) = 1$.
Therefore $t(\overline{1}) = 1$. By  Lemma~\ref{lem:criteria} then
$\FV(n)/\psi$ is projective as claimed. 
\end{proof}

Note that for any subquasivariety $\mathcal{Q}$ of a variety
$\mathcal{V}\subseteq\var{Ho}$, if all finite subdirectly irreducibles from
$\mathcal{V}$ belong to $\mathcal{Q}$, then $\mathcal{Q} = \mathcal{V}$.
Therefore, to test structural completeness it suffices to consider finite
subdirectly irreducibles.

\begin{theorem}\label{thm:generators}
Let $\mathcal{V}\subseteq\var{Ho}$. If every finite subdirectly irreducible
$\alg{A}\in\mathcal{V}$ has a set of generators $\{a_1,\dots, a_n\}$, where
$a_1 = \op$, such that
$\alg{A}\models \{a_2,\dots,a_n\}\ra 0 \neq 1$, then
$\mathcal{V}$ is hereditarily structurally complete.
\end{theorem}  

\begin{proof}
Let $f\colon\FV(n)\to \alg{A}$ be the homomorphism extending the natural map
given by $x_i\mapsto a_i$. By assumption
$f(\ell) \neq 1$ in $\alg{A}$, and thus taking
$\varphi = \ker{f}$ we get that $\FV(n)/\varphi \cong \alg{A}$
and $\ell\notin 1/\varphi$. By Lemma~\ref{lem:unique-term} we get that
$\alg{A}$ is projective. Since $\alg{A}$ was arbitrary,
all finitely generated free algebras in $\mathcal{V}$ are projective.
Therefore $\mathcal{V}$ is hereditarily structurally complete.
\end{proof}

\begin{example}
Let $\mathcal{V}\subseteq \var{Ho}$ be such that every finite subdirectly irreducible
algebra in $\mathcal{V}$ is an ordinal sum $\alg{B}\oplus\alg{A}$, where
$\alg{B}$ is a Boolean algebra, and $\alg{A}$ is any subdirectly irreducible
Hilbert algebra. Since $\alg{B}$ is a Boolean algebra, we can find a set
$\{b_1,\dots,b_k\}$ of generators of $\alg{B}$ such that
$\{b_1,\dots, b_k\}\ra 0 \neq 1$ (note that this means $b_1\wedge\dots\wedge
b_k\neq 0$ in the language of Boolean algebras). Now, let $a_1,\dots, a_s$ be
elements of $A$ such that $b_1,\dots,b_k,a_1,\dots,a_s$ generate  
$\alg{B}\oplus\alg{A}$. Since $a_i\ra 0 = 0$ for all $a_i$, we get that
$\{b_1,\dots,b_k,a_1,\dots,a_s\}\ra 0 \neq 1$. Hence, by
Theorem~\ref{thm:generators} the variety $\mathcal{V}$ is hereditarily
structurally complete. Such varieties include
Hilbert algebras with zero of height 2, that is,
subvarieties of $\var{Ho}$ satisfying the equation \textup($H_2$\textup);
linear Hilbert algebras, that is, $\var{Hb_1}$; and the variety of
$\{0,\ra\}$-subreducts of the algebras satisfying the weak excluded middle law
$\neg x\vee \neg\neg x = 1$. 
\end{example}

\section{Appendix: algorithms}\label{sec:append}
Algorithms~1, and~2 (see below) construct posets $\Cm{\alg{F}_{\var{Hi}}(n)}$,
and $\Cm{\alg{F}_{\var{Ho}}(n)}$. Since $\Cm{\alg{F}_{\var{Br}}(n)}$ and
$\Cm{\alg{F}_{\var{Bo}}(n)}$ are, respectively, isomorphic to those, they work
for them as well without any changes. Inspecting the algorithms for
$\Cm{\alg{F}_{\var{Hi}}(n+1)}$ and $\Cm{\alg{F}_{\var{Ho}}(n)}$ reveals that
$\Cm{\alg{F}_{\var{Ho}}(n)}$ can also be obtained from
$\Cm{\alg{F}_{\var{Hi}}(n+1)}$ by removing the downset $\dw\{(L,i): x_{n+1}\in L\}$. 

{\small
\begin{algorithm}[h]
\caption{Constructing $\Cm{\alg{F}_{\var{Hi}}(n)}$}
\begin{algorithmic}
\Require{$J\subsetneq X$, $1\leq i\leq n$}
\Comment{Points are pairs $(J,i)$}  
\State For a set $U$ of points,
$\textsc{Proj}(U) \deq \{J:  (J,i)\in U \text{ for some } i\}$
\State Let $P_1(n) = \{(J,1): J\subsetneq X\}$; $\leq\ \leftarrow\ =$ 
\Function{BuildNextLayer}{$\bigcup_{i=1}^{k-1}P_i(n)$} 
\State $P_k(n)\leftarrow \emptyset$
\For{$G\in \mathrm{Up}\left(\bigcup_{r=1}^{k-1}P_r(n)\right)$,
$G\cap P_{k-1}(n)\neq \emptyset$,
$L \subsetneq \bigcap \textsc{Proj}(G)$}
\State $P_k(n)\leftarrow P_k(n)\cup\{(L, k)\}$
\For{$(J,j)\in G$}
\State $\leq\ \leftarrow\ \leq\cup\ \langle(L,k),(J,j)\rangle$
\EndFor
\EndFor
\EndFunction
\For{$k\in\{2,\dots,n\}$}
\State $\textsc{BuildNextLayer}(\bigcup_{i=1}^{k-1}P_i(n))$
\EndFor
\end{algorithmic}\label{alg:build-P}
\end{algorithm}
}%
Now to apply Theorem~\ref{thm:free-Hilbert-alg} or
Theorem~\ref{thm:free-Hi0-alg} to build the desired
free algebra we only need to describe the
antichain selection function $S$. Let $\mathbb{P}(n)$ be the poset constructed
by either algorithm. By straightforward calculations we get
$$
S(x) = \{(L,i)\in\mathbb{P}(n):
x\notin L \text{ and } x\in L'\text{ for all } (L',j)>(L,i)\}.
$$

{\small
\begin{algorithm}[h]
\caption{Constructing $\Cm{\alg{F}_{\var{Ho}}(n)}$}
\begin{algorithmic}
\Require{$L\subseteq X$, $1\leq i\leq n+1$}
\Comment{Points are pairs $(L,i)$}  
\State For a set $U$ of points,
$\textsc{Proj}(U) = \{L:  (L,i)\in U \text{ for some } i\}$
\State Let $P_1(n) = \{(J,1): J\subseteq X\}$; $\leq\ \leftarrow\ =$ 

$\vdots$
\For{$k\in\{2,\dots,n+1\}$}
\State $\textsc{BuildNextLayer}(\bigcup_{i=1}^{k-1}P_i(n))$
\EndFor
\end{algorithmic}\label{alg:build-P0}
\end{algorithm}
}

\bibliographystyle{plain}
\bibliography{hilbert-algebras-final}

\end{document}